\documentclass[a4paper, 10pt]{amsart}
\usepackage{amsmath}
\usepackage{amssymb, color, hyperref}

\usepackage[dvipsnames]{xcolor}

\theoremstyle{plain}
\newtheorem{theorem}{\bf Theorem}[section]
\newtheorem{proposition}[theorem]{\bf Proposition}
\newtheorem{lemma}[theorem]{\bf Lemma}
\newtheorem{corollary}[theorem]{\bf Corollary}

\theoremstyle{definition}

\newtheorem{definition}[theorem]{\bf Definition}

\newcommand{\N}{\mathbb N}

\newcommand{\F}{\mathbb F}

\renewcommand{\t}{\, | \,}

 \DeclareMathOperator{\ord}{ord}

 \DeclareMathOperator{\supp}{supp}

\renewcommand{\t}{\, | \,}

\numberwithin{equation}{section}

\begin{document}

\title{On a classical zero-sum invariant}

\author{Alfred Geroldinger and Wenkai Yang}

\address{Department of Mathematics and Scientific Computing\\ University of Graz, NAWI Graz\\ Heinrichstra{\ss}e 36\\ 8010 Graz, Austria}
\email{alfred.geroldinger@uni-graz.at}
\urladdr{https://geroldinger.github.io/}

\address{Center for Combinatorics \\ Nankai University, Tianjin, P.R. China}
\email{9820250116@nankai.edu.cn}



\thanks{This work was supported by the Austrian Science Fund FWF, Project Number P36852-N}

\maketitle

\begin{abstract}
Let $G$ be a nontrivial, finite abelian group. Then $\nu (G)$ is the smallest integer $\ell$ such that every zero-sum free sequence $T$ over $G$ of length at least $\ell$ has the following property: all nonzero elements of $G$ that do not occur as a subsequence sum of $T$ lie in a proper coset  of some subgroup of $G$. We study the  invariant $\nu (G)$, which was introduced in Zero-Sum Theory in the 1960s. 
\end{abstract}

\smallskip
\section{Introduction} \label{1}
\smallskip

Let $G$ be a finite abelian group and let $G^{\bullet} = G \setminus \{0\}$. The (small) Davenport constant $\mathsf d ( G)$ is the maximal length of a zero-sum free sequence over $G$, and the (large) Davenport constant $\mathsf D (G)$ is the maximal length of a minimal zero-sum sequence over $G$.  The study of the structure of long minimal zero-sum sequences $S$ (resp. of long zero-sum-free sequences) and the structure of the corresponding set $\Sigma (S)$ of subsequence sums is a central topic in Zero-Sum Theory. The structure of long zero-sum free sequences over cyclic groups is given by the Theorem of Chen-Savchev-Yuan and the structure of minimal zero-sum sequences of length $\mathsf D (G)$ over rank two groups can be found in \cite[Chapter 4]{Ge-Gr-Zh26a}. Only sporadic structural results for sequences that are extremal with respect to given properties as above are known for groups of higher rank (e.g., \cite{Gi08b,Sc12a, Gi-Sc20a,Zh21b,Fa-Hu-Zh24a,Hu-Hu-Li24,Zh26a} for some recent contributions), and even a seemingly innocent question, whether every zero-sum free sequence of length $\mathsf d (G)$ has an element whose order is the exponent of the group, is open. 

The following invariant $\nu (G)$ provides insight into  the structure of subsequence sums of  extremal zero-sum free sequences. We suppose $|G|>1$ and start with a 
simple observation.
For a zero-sum free sequence $T$ of length $\mathsf d (G)$, the  set $\Sigma (T)$ of subsequence sums satisfies $\Sigma (T) = G^{\bullet}$. The invariant $\nu (G)$ goes a step further. Indeed, $\nu (G)$  is defined
to be the smallest integer such that, for any zero-sum free sequence $T$ over $G$,
\begin{equation} \label{def-nu(G)}
|T|\geq \nu(G)  \ \mbox{ implies } \ G^{\bullet} \setminus \Sigma (T) \subseteq \alpha+H \ \mbox {for some  subgroup $H  \subsetneq G$ and some $\alpha\in G\setminus H$}.
\end{equation}
In the 1960s (\cite[page 15]{Em69a}) van Emde Boas introduced $\nu (G)$   in order to study the Davenport constant of rank three groups (for this interplay see \cite{Ga00b}). It is easy to verify (\cite[Proposition 3.2.5]{Ge-Gr-Zh26a}) that 
\begin{equation} \label{basic-inequ-1}
\mathsf d (G) -1 \le \nu (G) \le \mathsf d (G) \,,
\end{equation}
and  that for a cyclic group we have $|G|-2=\mathsf d (G)-1= \nu (G)$. In contrast to this simple observation, the further study of $\nu (G)$ has proven to be quite challenging. Group algebra methods can be used to show that for $p$-groups we have $\mathsf d (G)-1=\nu (G)$ (\cite[Theorem 5.5.9]{Ge-HK06a}). If $G$ has rank two, then the first unconditional proof showing that $\mathsf d (G)-1= \nu (G)$ was  given only recently in \cite[Theorem 3.4.11]{Ge-Gr-Zh26a}.  Almost 30 years ago, it was formulated as a conjecture (Gao, \cite[Section 3]{Ga00b}) that $\nu (G) = \mathsf d (G)-1$ holds for all nontrivial finite abelian groups. The conjecture is still open and seems to be out of reach (see Theorem \ref{6.4} and the beginning of Section \ref{6}). Indeed, $p$-groups, cyclic groups, and rank two groups are the only series of groups  for which the precise value of $\nu (G)$ has been determined.

In Section \ref{3} we introduce a finer variant of $\nu (G)$ (denoted by $\nu_p (G)$) that puts an additional condition on the subgroup $H$ occurring in \eqref{def-nu(G)}. We determine the precise value of this variant for cyclic groups and for $p$-groups (Proposition \ref{3.3} and Theorem \ref{3.5}). 
Cyclic groups and elementary $2$-groups are extremal cases for finite abelian groups, and many zero-sum problems are studied for these groups first. Groups of the form $G = C_2^r \oplus C_{2n}$ are the direct sums of these extremes and they have received considerable attention in Zero-Sum Theory. From Section \ref{4} on,  we study  $\nu (G)$  for groups of this type. In Theorem \ref{4.7}, we show that $\mathsf d (G)-1= \nu_2 (G) = \nu (G)$ for $G = C_2^2 \oplus C_{2n}$.
So far, $\nu (G)$ has been studied only for groups $G$ with $\mathsf d (G) =\mathsf d^* (G)$ (see \eqref{davenport}). In Theorem \ref{5.5}, we show that $\mathsf d (G)-1= \nu_2 (G) = \nu (G)$ for $G = C_2^4 \oplus C_{2n}$ with $n > 70$ odd. These  groups $G$ have the property $\mathsf d (G) = 1+ \mathsf d^* (G)$. In Section \ref{6}, we introduce a local variant of $\nu (G)$ which offers a greater flexibility in the  study of subsequence sums of extremal zero-sum free sequences.

\smallskip
\section{Background on  sequences over finite abelian groups} \label{2}
\smallskip

Let $G$ be an additively written, finite abelian group, and let $G_0 \subseteq G$ be a subset. Then $\langle G_0 \rangle \subseteq G$ denotes the subgroup generated by $G_0$ and $\mathcal F (G_0)$ denotes the multiplicative free abelian monoid with basis $G_0$. An element 
\[
S = g_1 \cdot \ldots \cdot g_{\ell} = \prod_{g \in G} g^{\mathsf v_g (S)} \in \mathcal F (G_0)
\]
is called a {\it sequence} over $G_0$, and we denote by
\[
\begin{aligned}
|S| & = \ell = \sum_{g \in G} \mathsf v_g (S) \in \N_0 \quad \text{the length of $S$, } \\
\mathsf h (S) & = \max \{ \mathsf v_g (S) \colon g \in G\} \in \N_0 \quad \text{the maximum multiplicity of a term of $S$, } \\
\sigma (S) & = g_1 + \ldots + g_{\ell} \quad \text{the sum of $S$,  } \\
\Sigma (S) & = \Big\{ \sum_{i \in I} g_i \colon \emptyset \ne I \subseteq [1, \ell] \Big\} \quad \text{the set of subsequence sums of $S$,} \\
\Sigma^*(S) & = \Sigma (S) \cup \{0\}.
\end{aligned}
\]
Furthermore, we set
\[
\begin{aligned}
\varphi (S) & = \varphi (g_1) \cdot \ldots \cdot \varphi (g_{\ell}) \quad  \text{for any map
$\varphi \colon G \to G'$ and any group $G'$}, \\
-S & = (-g_1) \cdot \ldots \cdot (-g_{\ell}), \\
g + S & = (g+g_1) \cdot \ldots \cdot (g+g_{\ell}) \quad  \text{for every $g \in G$}, \quad \text{and} \\
S_X & = \prod_{g \in X} g^{\mathsf v_g (S)} \quad  \text{for every subset $X \subseteq G$} \,.
\end{aligned}
\]
We say that $S$ is {\it zero-sum free} if $0 \notin \Sigma (S)$, a {\it zero-sum sequence} if $\sigma (S)=0$, and a  {\it minimal zero-sum sequence} if $\sigma (S) = 0$ but every proper subsequence is zero-sum free.  We denote by
\begin{itemize}
\item $\mathcal A (G)$ the set of all minimal zero-sum sequences over $G$, by 

\item $\mathcal A_{\max} (G)$ the set of all minimal zero-sum sequences over $G$ of maximal length,  by

\item $\mathsf D (G) = \max \{ |S| \colon S \in \mathcal A (G) \}$ the {\it (large) Davenport constant} of $G$,  and by

\item $\mathsf d (G) = \max \{ |S| \colon S \in \mathcal F (G) \ \text{zero-sum free} \}$ the   {\it (small) Davenport constant} of $G$.
\end{itemize}
If $G \cong C_{n_1} \oplus \ldots \oplus C_{n_r}$, with $1 < n_1 \t \ldots \t n_r$, then $r=\mathsf r (G)$ is the {\it rank} of $G$, and we set
\[
\mathsf d^* (G) = \sum_{i=1}^r (n_i-1) \quad \text{and} \quad \mathsf D^* (G) = 1 + \mathsf d^* (G) \,.
\]
We have the well-known inequalities
\begin{equation} \label{davenport}
\mathsf d^* (G) \le \mathsf d (G) \quad \text{and} \quad \mathsf D^* (G) \le \mathsf D (G) = \mathsf d (G)+1 \,.
\end{equation}
Equality holds for $p$-groups and for groups with rank $r \le 2$. These are the only groups for which $\nu (G)$ has been determined so far, and for them  $\nu (G) = \mathsf d (G)-1$ holds (see \cite[Theorem 4.5.6]{Ge-Gr-Zh26a}). However, there are various series of groups for which the above inequalities for the Davenport constant are strict (\cite{Ge-Li-Ph12, Li20a, Zh23a}), and we study the $\nu (G)$ invariant for one such series in Section \ref{5}.
No group is known so far with $\nu (G) = \mathsf d (G)$.

\begin{proposition} \label{2.2}
Let $G = C_2^r \oplus C_{2n}$ with $r,\, n\geq 1$.
\begin{enumerate}
\item 
If $r \le 2^{\mathsf v_2(n)+1} - 1$, then $\mathsf d (G) = \mathsf d^* (G)$. In particular, $\mathsf d (C_2^r\oplus C_{2n})=\mathsf d^*(C_2^r\oplus C_{2n})$ for all $n\geq 2$ even and $r\leq 3$.

\item If   $n\geq 3$ is odd and $r\geq 1$, then $\mathsf d (G) = \mathsf d^* (G)$ if and only if $r \le 3$.
\end{enumerate}
\end{proposition}

\begin{proof}
See \cite[Theorems 3.9.2 and  3.9.7]{Ge-Gr-Zh26a}.
\end{proof}

We continue with a series of simple lemmas that will be used in all further sections. Let $G$ be a nontrivial, finite abelian group.

\begin{lemma}\label{2.3}
Let $T \in \mathcal F (G)$ be zero-sum free.
\begin{enumerate}
\item  If  $|T|=\mathsf d (G)$, then $G^{\bullet} \setminus \Sigma (T)= \emptyset$ and $S = (-\sigma (T))T \in \mathcal A_{\max}(G)$.

\item If $|G^{\bullet} \setminus \Sigma (T)| \le 1$, then $G^{\bullet} \setminus \Sigma (T)$ is contained in the coset of a proper subgroup.

\item If $G^{\bullet}\setminus \Sigma (T)=\{g,h\}$ with $g \notin \langle h-g \rangle$, then $G^{\bullet}\setminus \Sigma (T)\subseteq g+\langle h-g \rangle$. 

\item If  $G^{\bullet}\setminus \Sigma (T)=\{g,h\}$ with $g, h  \in \langle h-g \rangle$, then there is no proper subgroup $H \subsetneq G$ and no $\alpha\in G\setminus H$  such that $G^{\bullet}\setminus \Sigma (T)=\{g,h\}\subseteq \alpha+H $.
\end{enumerate}  
\end{lemma}

\begin{proof}
1. 	Assume to the contrary that there is some $g \in G^{\bullet} \setminus \Sigma (T)$. Then $(-g)T$ is zero-sum free of length $|(-g)T| > |T| = \mathsf d (G)$, a contradiction. By definition, $S = (-\sigma (T))T $ is a  zero-sum sequence of length $\mathsf d (G)+1 = \mathsf D (G)$. If a proper subsequence of $S$ containing $-\sigma(T)$ were zero-sum, then its complement in $T$ would be a nontrivial zero-sum subsequence of $T$; hence $S$ is minimal zero-sum.

2. and 3. Obvious. 

4. Assume to the contrary that $G^{\bullet}\setminus \Sigma (T)=\{g,h\}\subseteq \alpha+H $ for some subgroup $H  \subsetneq G$ and some $\alpha\in G\setminus H$. Then $g+H=\alpha+H = h+H$,  whence $h-g \in H $ which implies  $g, h  \in \langle h-g \rangle\subseteq  H$, a contradiction to $\alpha\in G\setminus H$.
\end{proof}

\begin{lemma}\label{2.4}
Let $U\in \mathcal A(G)$ with $|U|=\mathsf d (G)$. Then for every $g \in \supp ( U)$, one of the following properties holds.
\begin{enumerate}
\item[(a)]  There is  some $U_0\in \mathcal A_{\max}(G)$ such that $Ug^{-1}\mid U_0$.
		
\item[(b)] $\Sigma (Ug^{-1})=G^{\bullet}$.
\end{enumerate}
\end{lemma}

\begin{proof}
	Assume that $\Sigma (Ug^{-1})\ne G^{\bullet}$. Then there is some $h \in G^{\bullet}$ with $-h\in G^{\bullet}\setminus \Sigma (Ug^{-1})$. This implies that $Ug^{-1}h\in \mathcal F (G)$ is zero-sum free and  $|Ug^{-1}h|=\mathsf d (G)$. We set $h_0 = - \sigma ( Ug^{-1}h )$. Then Lemma  \ref{2.3}.1 shows that $U_0 = h_0 U g^{-1}h \in \mathcal A_{\max} (G)$.
\end{proof}

\begin{lemma}\label{2.5}
Let $U\in \mathcal A_{\max}(G)$ and let $g_1, g_2 \in G$ with $g_1g_2\mid U$. 
\begin{enumerate}
\item For every $h\in G^{\bullet}$ the following properties are equivalent. 
      \begin{enumerate}
	  \item[(a)] $-h\in \Sigma(Ug_1^{-1}g_2^{-1})$. 

      \item[(b)] There is no  $U'\in \mathcal A_{\max}(G)$ such that $hUg_1^{-1}g_2^{-1}\mid U'$.
	  \end{enumerate}

\item $-g_1,-g_2\in G^{\bullet}\setminus \Sigma(Ug_1^{-1}g_2^{-1})$.
\end{enumerate}
\end{lemma}

\begin{proof}
1.	Let $h \in G$.

(a) $\Longrightarrow$ (b)  If  $-h\in \Sigma(Ug_1^{-1}g_2^{-1})$, then there exists a nontrivial subsequence $T$ of $ Ug_1^{-1}g_2^{-1}$ such that $\sigma(T)=-h$, whence $hT\mid hUg_1^{-1}g_2^{-1}$ and hence $hUg_1^{-1}g_2^{-1}$ cannot be a subsequence of  a minimal zero-sum sequence.

(b) $\Longrightarrow$ (a) 
Since $|hUg_1^{-1}g_2^{-1}|=\mathsf d (G)$ and $hUg_1^{-1}g_2^{-1}$ is not a subsequence of a minimal zero-sum sequence, it follows that  $hUg_1^{-1}g_2^{-1}$ is not zero-sum free. Thus,  $0\in \Sigma(hUg_1^{-1}g_2^{-1})$ whence $-h\in \Sigma(Ug_1^{-1}g_2^{-1})$. 
	
2. Let $i \in [1,2]$ and
assume to the contrary that $-g_i\in \Sigma(Ug_1^{-1}g_2^{-1})$, say $i=1$. Part 1 implies that $g_1Ug_1^{-1}g_2^{-1}= Ug_2^{-1}$ is not a subsequence of a minimal zero-sum sequence, a contradiction to $U\in \mathcal A_{\max}(G)$.
\end{proof}

\begin{proposition} \label{2.6}
Let $T\in \mathcal F (G)$ be zero-sum free with $|T|=\mathsf d (G)-1$. Then
\[
\begin{aligned}
G^{\bullet}\setminus \Sigma (T) & =\bigcup_{ U\in \mathcal A_{\max}(G), T \mid U } \supp \big(-(UT^{-1}) \big) \\
	& = \ \{-g\in G \colon \text{there are $U\in \mathcal A_{\max}(G)$ and } \ gT \mid U \} \,.
\end{aligned}
\]
In particular,  $G^{\bullet}\setminus \Sigma (T)=\emptyset$ if and only if there is no  $U\in \mathcal A_{\max}(G)$ with $T \mid U$.
\end{proposition}

\begin{proof}
First we show that $G^{\bullet}\setminus \Sigma (T)  \subseteq \bigcup_{ U\in \mathcal A_{\max}(G), T \mid U } \supp \big(-(UT^{-1}) \big)$. Let $g \in G$ with $-g \in G^{\bullet}\setminus \Sigma (T)$. Then $gT$ is zero-sum free of length $|gT|= \mathsf d (G)$, whence $U = ghT \in \mathcal A_{\max} (G)$  with $h = - \sigma (gT)$. Thus $-g \in \supp \big(-(UT^{-1}) \big)$.

Conversely, we show that $\bigcup_{ U\in \mathcal A_{\max}(G), T \mid U } \supp \big(-(UT^{-1}) \big) \subseteq G^{\bullet}\setminus \Sigma (T)$. Let $U \in \mathcal A_{\max} (G)$ with $U = Tgh$ whence $|U|=|T|+2$. Then $\supp \big(-(UT^{-1}) \big) = \{-g,-h\}$, and Lemma \ref{2.5}.2 implies that $-g,-h \in G^{\bullet}\setminus \Sigma (T)$.

The second equality is clear, and the same is true for the "In particular" statement.
\end{proof}

\smallskip
\section{On a refinement of $\nu (G)$} \label{3}
\smallskip

In this section we introduce a finer version of $\nu (G)$ and determine its precise value for cyclic groups and for $p$-groups (Proposition \ref{3.3} and Theorem \ref{3.5}).

\begin{definition} \label{3.1}
Let $G$ be a finite abelian group and let $p$ be a prime divisor of $|G|$. Let $\nu_p (G)$ be the smallest integer such that, for any zero-sum free sequence $T$ over $G$, $|T| \ge \nu_p (G)$ implies that 
\begin{equation} \label{def-nu_p(G)}
G^{\bullet} \setminus \Sigma (T) \subseteq \alpha + H \ \text{for some subgroup} \ H \subseteq G \ \text{with} \ (G \colon H) = p \ \text{and some} \ \alpha \in G \setminus H \,.
\end{equation}
\end{definition}

Since $\Sigma (T) = G^{\bullet}$ for every zero-sum free sequence of length $|T|= \mathsf d (G)$, \eqref{basic-inequ-1} implies that
\begin{equation} \label{basic-inequ-2}
\mathsf d (G) -1 \le \nu (G) \le \nu_p (G) \le \mathsf d (G) \,.
\end{equation}
Van Emde Boas \cite{Em69a} introduced the notion that a group of even order has Property Q if $\nu_2 (G) = \mathsf d (G)-1$. He verified Property Q for several groups, and
no group of even order is known so far with $\nu_2 (G) = \mathsf d (G)$. 

\medskip
A central problem concerning the structure of minimal zero-sum sequences of maximal length concerns the order of the elements in such sequences. A long-standing open conjecture states that every such extremal sequence has at least  one element whose order equals the exponent of the group (see \cite[Appendix B, Problem 4]{Ge-Gr-Zh26a}). Let $G =C_n^r$ with $r, n \in \N$ and $n \ge 2$. If $r \ge 3$, then not only is the precise value of $\mathsf d (C_n^r)$ open (for general $n$; for recent progress see \cite{Gi18a}) but there is not even a conjecture concerning the structure of minimal zero-sum sequences of maximal length (the case $r=2$ is settled in \cite[Chapter 4]{Ge-Gr-Zh26a}). The next result shows how the value of $\nu_p (G)$ is related to questions concerning the order of elements occurring in minimal zero-sum sequences of maximal length.

\begin{proposition} \label{3.2}
Let $G $ be a finite abelian group.   
\begin{enumerate}
\item Let $p$ be a prime divisor of $|G|$. For each $g\in G\setminus pG$,  there is a subgroup $K$ with  $ (G:K)=p$ and $-g \notin  K$. If $\nu_p(G)=\mathsf d (G)-1$, then $\supp (U) \cap pG = \emptyset$ for every $U \in \mathcal A_{\max} (G)$.

\item Let $G = C_n^r$,  with $n, r \in \N$ and $n \ge 2$, and suppose that $\nu_p (G) = \mathsf d (G)-1$ for every prime divisor $p$ of $|G|$. Then for every $U \in \mathcal A_{\max} (G)$ and for every $g \in \supp (U)$ we have $\ord (g) = \exp (G)$.
\end{enumerate}
\end{proposition}

\begin{proof}
1. Let $p$ be a prime divisor of $|G|$. The existence of the subgroup $K$ is clear.  Assume to the contrary that there is some  $U\in \mathcal A_{\max} (G)$ with $\supp (U) \cap pG \ne \emptyset$,  say $g \in \supp (U) \cap pG$. Let $T \in \mathcal F (G)$ be zero-sum free with $|T| = \mathsf d (G)-1$, $T \mid U$ and $g\mid UT^{-1}$.
Then by Proposition \ref{2.6}, $-g\in G^{\bullet}\setminus \Sigma (T)$. 
	
Note that $pG$ is the intersection of all subgroups $K \subseteq G$ with $(G : K)=p$. Since $-g\in pG$, $-g\in K$ for every subgroup $K$ with  $ (G:K)=p$. Thus there is no subgroup $K$ with  $ (G:K)=p$ and $\alpha \in G\setminus K$ such that $G^{\bullet}\setminus \Sigma (T)\subseteq \alpha +K$ whence $\nu_p (G) = \mathsf d (G)$, a contradiction.

2. Assume to the contrary that there are  $U \in \mathcal A_{\max} (G)$ and  $g \in \supp (U)$ with $\ord (g) < \exp (G)$. Then there are a prime $p$ and some $g' \in G$ with $pg' = g$, whence $g \in \supp (U) \cap pG$, a contradiction to Part 1.
\end{proof}

\begin{proposition} \label{3.3}
Let $G$ be a finite cyclic group with $|G| = n \ge 2$, and let $p$ be a prime divisor of $|G|$. Then $\nu (G) = \nu _p (G) = \mathsf d (G)-1$.
\end{proposition}

\begin{proof}
By \eqref{basic-inequ-2}, it suffices  to prove that $\nu_p (G) \le \mathsf{d}(G) - 1$. For this we need to prove that $\mathsf d (G)-1$ satisfies \eqref{def-nu_p(G)}.
Let $H$ denote the unique subgroup of $G$
	of index $p$. Let $T \in \mathcal{F}(G)$ be zero-sum free with
	$|T| = \mathsf{d}(G) - 1 = n - 2$.
	
	If $G^\bullet \setminus \Sigma(T) = \emptyset$, the claim is trivial. Otherwise
	pick $x \in G^\bullet \setminus \Sigma(T)$. Then $(-x) T$ is a zero-sum free
	sequence of length $n - 1 = \mathsf{d}(G)$, whence, by 
	\cite[Theorem~5.1.10]{Ge-HK06a}, there is some   $e \in G$ such that
	\[
	(-x) T \;=\; e^{n-1}.
	\]
This shows that  $x = -e$ and $T = e^{n-2}$, whence
	\[
	\Sigma(T) \;=\; \{e, 2e, \ldots, (n-2)e\}
	\qquad\text{and}\qquad
	G^\bullet \setminus \Sigma(T) \;=\; \{-e\}.
	\]
	Since $e$ generates $G$, the image of $e$ in $G/H \cong C_p$ is nonzero, and
	hence $-e \notin H$. Therefore
	\[
	G^\bullet \setminus \Sigma(T) \;\subseteq\; -e + H,
	\]
	which proves $\nu_p (G) \le \mathsf{d}(G) - 1$.
\end{proof}

Let $p$ be a prime and let $G$ be a finite abelian $p$-group. We denote by $\F_p$ the finite field with $p$ elements and by  $\F_p[G]$  the group algebra of $G$ over $\F_p$. We denote by  $\{X^g \colon g \in G \}$  an $\F_p$-basis of $\F_p [G]$, whence every element $f \in \F_p [G]$ can be uniquely written in the form
\[
f = \sum_{g \in G} c_gX^g, \ \text{where} \ c_g \in \F_p \ \text{for all $g \in G$} \,.
\]
For 
$S = g_1 \cdot \ldots \cdot g_\ell \in \mathcal{F}(G)$, we define
\[
\Pi(S) \;=\; \prod_{i=1}^{\ell}(1 - X^{g_i})=\Sigma_{g\in G} c_g X^g \;\in\; \mathbb{F}_p[G].
\]

\begin{lemma} \label{3.4}
Let $G$ be a finite abelian $p$-group with $|G| \ge p$ and let $S \in \mathcal F (G)$.
\begin{enumerate}
\item If $|S| \ge \mathsf d (G)+1$, then $\Pi (S) = 0$.

\item If $|S| = \mathsf d (G)$, then there is some $c \in \F_p$ such that $\Pi (S) = c \sum_{g \in G} X^g$.
\end{enumerate}
\end{lemma}

\begin{proof}
See \cite[Proposition 5.5.8]{Ge-HK06a}.
\end{proof}

\begin{theorem} \label{3.5}
Let $G$ be a finite abelian $p$-group with $|G| \ge p$. Then $\nu (G) = \nu_p (G) = \mathsf d (G)-1$.
\end{theorem}

\begin{proof}
By \eqref{basic-inequ-2}, it suffices  to prove that
	$\nu_p (G) \le \mathsf{d}(G) - 1$. For this we need to prove that $\mathsf d (G)-1$ satisfies \eqref{def-nu_p(G)}.
Let
	$T \in \mathcal{F}(G)$ be zero-sum free with $|T| = \mathsf d (G) - 1$. We must
	prove that $G^\bullet \setminus \Sigma(T)$ is contained in a proper coset
	of some subgroup of $G$ of index $p$.
	If $G^\bullet \setminus \Sigma(T) = \emptyset$, the
	claim is clear. So we suppose that  $G^\bullet \setminus \Sigma(T) \ne \emptyset$.
	
	For every $y \in G$, the sequence $yT$ has length $\mathsf d (G)$. Hence, by
	Lemma~\ref{3.4}, there exists a unique
	$\lambda(y) \in \mathbb{F}_p$ such that
	\begin{align} \label{def-lambda-Pi}
	(1 - X^y) \Pi(T) \;=\; \lambda(y) \sum_{g \in G} X^g.
	\end{align}
Thus we obtain a map $\lambda \colon  G \to \mathbb{F}_p$, and we assert that 
$\lambda$ is a group homomorphism. For $y, z \in G$, the
	identity
	\[
	1 - X^{y+z} \;=\; (1 - X^y) + (1 - X^z) - (1 - X^y)(1 - X^z)
	\]
	gives, after multiplication by $\Pi(T)$,
	\[
	(1 - X^{y+z}) \Pi(T)
	\;=\;
	(1 - X^y) \Pi(T) + (1 - X^z) \Pi(T),
	\]
	because
	$(1 - X^y)(1 - X^z) \Pi(T) = \Pi(yzT) = 0$ by
	Lemma~\ref{3.4} (since $|yzT| = \mathsf d (G) + 1$). Comparing the
	coefficients of \ $\sum_{g \in G} X^g$ \ in the above equality, via
	\eqref{def-lambda-Pi}, yields
	\[
	\lambda(y + z) \;=\; \lambda(y) + \lambda(z).
	\]
Thus, $\lambda$ is a group homomorphism.
	
Next, let  $x \in G^\bullet \setminus \Sigma(T)$. We set $S = (-x)T=g_1 \cdot \ldots \cdot g_{\ell}$ and $S (I) = \prod_{i \in I} g_i$ for every $\emptyset \ne I \subset [1, \ell]$. Since  $S$ is zero-sum free, we obtain that
\[
\begin{aligned}
\Pi \big( (-x) T \big) = & \Pi (S) = \prod_{i=1}^{\ell} (1 - X^{g_i}) = 1 + \sum_{\emptyset \ne I \subseteq [1, \ell]} (-1)^{|I|} X^{\sigma \big( S (I) \big)} \,.\\
\end{aligned}
\]
On the other hand, by \eqref{def-lambda-Pi},
	\[
	\Pi \big( (-x) T\big) \;=\; \lambda(-x) \sum_{g \in G} X^g.
	\]
	Since the coefficient of $X^0$ in $\sum_{g \in G} X^g$ is also $1$, we
	obtain $\lambda(-x) = 1$, and hence $\lambda(x) = -1$.
	
	In particular, $\lambda \ne 0$, and we set $H = \ker ( \lambda )$. Then  $|G| = |H|(G \colon H) = |H| |\lambda (G)|$ implies that $(G \colon H) = p$. 
Next we choose  any $\alpha \in G$ with
	$\lambda(\alpha) = -1$.
	Then $\alpha \notin H$, and for every $x \in G^\bullet \setminus \Sigma(T)$ we have
	$\lambda(x - \alpha) = \lambda(x) - \lambda(\alpha) = 0$, so
	$x - \alpha \in H$ and thus $x \in \alpha + H$. Consequently, we obtain that
	\[
	G^\bullet \setminus \Sigma(T) \;\subseteq\; \alpha + H \,. \qedhere
	\]
\end{proof}

The statement of the next corollary was known  for $r \le 2$. If $r \ge 3$, it was known that there exists an element $g \in \supp (U)$ with $\ord (g) = \exp (G)$ (see \cite[Appendix B, Problem 4]{Ge-Gr-Zh26a}).

\begin{corollary} \label{3.6}
Let $G = C_{p^k}^r$, where $p$ is a prime and $k, r \in \N$, and let $U \in \mathcal A_{\max} (G)$. Then $\ord (g) = p^k$ for every $g \in \supp (U)$.
\end{corollary}

\begin{proof}
This follows from Proposition \ref{3.2}.2 and Theorem \ref{3.5}
\end{proof}

\smallskip
\section{On groups of the form $C_2 \oplus C_2 \oplus C_{2n}$} \label{4}
\smallskip

The main goal in this section is to prove that $\nu (C_2^2 \oplus C_{2n}) = \nu_2 (C_2^2 \oplus C_{2n}) = \mathsf d (C_2^2 \oplus C_{2n}) - 1$ (Theorem \ref{4.7}). This result was first announced by Schmid, but no proof was given (\cite[Lemma  4.3]{Sc11b}). We start with groups of rank two.

\begin{lemma}\label{4.1}
Let $G = C_2 \oplus C_{2n}$ with $n\geq 2$.
\begin{enumerate}
\item A sequence $U \in \mathcal F (G)$ is a minimal zero-sum sequence of length $\mathsf D (G)$ if and only if there is a basis $(e_1, e_2)$ of $G$ with $\ord (e_1)=2$ and $\ord (e_2) = 2n$ such that $U$ is of one of the following types.
    \begin{enumerate}
	\item[(a)] $U  = e_2^{2n - 1}  (e_1 + a e_2)  (e_1 + (1 - a) e_2) \quad \text{with } a \in [0,  2n - 1]$.

	\item[(b)] $ U  = e_1  e_2^v (e_1 + e_2)^{2n - v} \quad \text{with } v \text{ odd and }  v \in [3, 2n - 3]$.
\end{enumerate}

\smallskip
\item $\nu (G) = \nu_2 (G) = \mathsf d (G)-1$.
\end{enumerate} 
\end{lemma}

\begin{proof}
1. See  \cite[Theorem 3.3]{Ga-Ge02} or \cite[Corollary 4.5.8]{Ge-Gr-Zh26a}.
	
2. By \eqref{basic-inequ-2}, it suffices to prove that $\nu_2 (G) \le \mathsf d (G)-1$. For this we need to prove that $\mathsf d (G)-1$ satisfies \eqref{def-nu_p(G)}. Let $T \in \mathcal F (G)$ be zero-sum free with $|T| = \mathsf d (G)-1$. We consider the set of all $U \in \mathcal A_{\max} (G)$ with $T \mid U$ and we distinguish three cases.

\smallskip
\noindent
CASE 1:   $|\{U \in \mathcal A_{\max}(G) \colon T \mid U \}|>1$, say $\{U \in \mathcal A_{\max}(G) \colon T \mid U \} \supseteq \{U, U'\}$. 

We distinguish three subcases.

\smallskip
\noindent
CASE 1.1: $U$ is of type (a) with respect to $(e_1, e_2)$ and $U'$ is of type (a) with respect to $(e_1', e_2')$.

If $|\supp(T)|=1$, then $T=e_2^{2n - 1}$ and $e_2=e_2'$. Then $\Sigma (T)=  \langle e_2 \rangle^{\bullet}$ and $G^{\bullet} \setminus \Sigma (T) \subseteq e_1 + \langle e_2 \rangle$.

If $|\supp(T)|=2$, then $\mathsf{v}_{e_2}(T)= (2n-1)-1\ge 2$ and $e_2=e_2'$. Let $T= e_2^{2n - 2} a$, then $U=U'=e_2^{2n - 1} a (e_2 -a)$, a contradiction.

If $|\supp(T)|=3$, say $\supp(T)=\{a,b,c\}$ with $a=b+c$, then $a=e_2$ and $U=U'=e_2^{2n - 1} bc$, a contradiction.

\smallskip
\noindent
CASE 1.2:  $U$ is of type (b) with respect to $(e_1, e_2)$ and $U'$ is of type (b) with respect to $(e_1', e_2')$.

Since $\mathsf{v}_{e_2}(U) \ge 3$ and $\mathsf{v}_{e_1+e_2}(U) \ge 3$, it follows that $\supp(T) \supseteq \{e_2, e_1+e_2 \}$ and $e_1= e_2 - (e_1+e_2)$, whence $\supp(U)=\supp(U')=\{e_1, e_2, e_1+e_2 \}$.

 If $e_1\mid T$, then $G^{\bullet} \setminus \Sigma (T)\subseteq \{-e_2, -e_1-e_2\} =-e_2+ H$ where $H= \langle e_1,2e_2 \rangle$ with $(G \colon H)=2$. 

 If $e_1\nmid T$, then either
$G^{\bullet} \setminus \Sigma (T)=\{e_1, e_1-e_2\}\subseteq e_1 +\langle e_2 \rangle $ or $\{-e_2, e_1\}\subseteq e_1 +\langle e_1 +e_2 \rangle$, and both subgroups have index $2$.

\smallskip
\noindent
CASE 1.3: $U$ is of type (a) with respect to $(e_1, e_2)$ and $U'$ is of type (b) with respect to $(e_1', e_2')$.

	Since $2n - 3= \mathsf h(U)-2\le \mathsf h(T) \le \mathsf h(U')\le 2n-3$, it follows that   $\mathsf h(T)=2n-3$, whence 
\[
T= e_2^{2n - 3}  (e_1 + a e_2)  (e_1 + (1 - a) e_2) 
\] 
with $e_1 + a e_2\ne e_1 + (1 - a) e_2$ for some $a \in [0,  2n - 1]$. 
	Since $T\mid U'$ and $|\supp (U')|=|\supp (T)|=3$, it follows that $e_1'\mid T$ with $\ord (e_1')=2$. This implies that $e_1'=e_1 + a e_2$ or $e_1'=e_1 + (1 - a) e_2$. So $a \in [0,1,n,1+n]$, whence $T= e_1'   e_2^{2n - 3}   (e_1' +  e_2)$ with $e_1'=e_1$ or $e_1'=e_1 + n e_2$. 
		It is easy to check that $G^{\bullet} \setminus \Sigma (T)=\{-e_2, -e_1'-e_2\} \subseteq -e_2+ \langle e_1' \rangle \subseteq -e_2+\langle e_1',2e_2 \rangle$.
	
\smallskip
\noindent
CASE 2:  $|\{U \in \mathcal A_{\max}(G) \colon T \mid U \}|=1$.

We distinguish two subcases.

\smallskip
\noindent
CASE 2.1: $U$ is of type (a) with respect to $(e_1, e_2)$.

By Proposition \ref{2.6}, $G^{\bullet} \setminus \Sigma (T)=\{g,h\}$ with $e_2^2\ne (-g)(-h)\mid U$. If $(-g)(-h)=e_2 (e_1 + a e_2)$ with $a \in [0,  2n - 1]$, then 
\[
\{g,h\} \subseteq 
\begin{cases} 
-e_2+ H_1 \quad \text{ with } H_1= \langle e_1,2e_2 \rangle \ \text{if $a $ is odd} \\
-e_2+ H_2 \quad \text{ with } H_2= \langle e_1+e_2 \rangle \ \text{if    $a $ is even} \,.
\end{cases}
\]
Note that $(G : H_1) = (G : H_2) = 2$.
	If $(-g)(-h)=(e_1 + a e_2)  (e_1 + (1 - a) e_2)$ with $a \in [0,  2n - 1]$, then $\{g,h\}\subseteq e_1+ H_3$ with $H_3= \langle e_2 \rangle$ and $(G \colon H_3)=2$.
	
\smallskip
\noindent
CASE 2.2: $U$ is of type (b) with respect to $(e_1, e_2)$. 

By Proposition \ref{2.6}, $G^{\bullet} \setminus \Sigma (T)=\{g,h\}$ with $ (-g)(-h)\mid U$ and $e_2^2\ne (-g)(-h) \ne (e_1+ e_2)^2$. Then $e_1\in \{g,h\}$,  say $h=e_1$. If $g=e_2$, then $\{g,h\}\subseteq e_1+ H_2$ with  $(G \colon H_2)=2$. If $g=e_1+ e_2$, then $\{g,h\}\subseteq e_1+ H_3$ with  $(G \colon H_3)=2$.

CASE 3:  $|\{U \in \mathcal A_{\max}(G) \colon T \mid U \}|=0$. 

This means any $S\in \mathcal A(G)$ with $T\mid S$ has length $|S|\le \mathsf d (G)$. Thus $S=-\sigma(T)T$ is the unique minimal zero-sum sequence, and Proposition \ref{2.6} gives $G^{\bullet} \setminus \Sigma (T)= \emptyset$.
\end{proof}

We continue with the characterization of minimal zero-sum sequences of maximal length in groups of the form $C_2 \oplus C_2 \oplus C_{2n}$.

\begin{proposition}[Schmid's Theorem] \label{4.2}
Let $G = C_2 \oplus C_2 \oplus C_{2n}$ with $n\geq 1$. A sequence $U \in \mathcal F (G)$ is a minimal zero-sum sequence of length $\mathsf D (G)$ if and only if there is a basis $(e_1, e_2, e_3)$ of $G$ with $\ord (e_1)=\ord (e_2)=2$ and $\ord (e_3) = 2n$ such  that $U$ is of one of the following types.
\begin{enumerate}
\item[(a)] $U = e_3^{v_3} (e_2+e_3)^{v_2} (e_1+e_3)^{v_1} (e_1+e_2-e_3)$, where  $v_3 \ge v_2 \ge v_1\geq 1$ are odd with  $v_3+v_2+v_1=2n+1$,

\item[(b)] $U = e_3^{v_3} (e_2+e_3)^{v_2} (e_1+xe_3) (e_1+e_2-xe_3)$, where  $v_3 \ge v_2\geq 1$ are odd,  $v_2+v_3=2n$, and $x \in [2, n-1]$,

\item[(c)] $U = e_3^{2n-1} (e_2+xe_3) (e_1+ye_3) (e_1+e_2+ze_3)$, where  $x, y \in [2, n-1]$, \ $z \in [2, 2n-3] \setminus \{n,n+1\}$, \  $x+y+z=2n+1$, and $ z\geq y\geq x$,

\item[(d)] $U=e_3^{2n-1-2v} (e_2+e_3)^{2v} e_2 (e_1+xe_3) \big(e_1+e_2+(1-x)e_3\big)$, where  $v \in [0, n-1]$ and $x \in [2, n-1]$,

\item[(e)] $U = e_3^{2n-2} (e_2+x e_3) \big( e_2+(1-x)e_3\big)  (e_1+ye_3)  \big( e_1+(1-y)e_3 \big)$, where  $x, y \in [2, n-1]$ and $x \ge y$,

\item[(f)] $U = e_1  e_2  \prod_{i=1}^{2n}(d_i+e_3)$ with $S=d_1  \cdot\ldots\cdot  d_{2n} \in \mathcal F (\langle e_1, e_2 \rangle)$ and $\sigma (S) = e_1+e_2$.
\end{enumerate}
\end{proposition}

\begin{proof}
See \cite[Theorem 3.13]{Sc11b} or \cite[Theorem 4.6.1]{Ge-Gr-Zh26a}.
\end{proof}

Let $G = C_2 \oplus C_2 \oplus C_{2n}$ with $n\geq 1$. We proceed in a series of lemmas needed for the proof of Theorem \ref{4.7}. Note, 
if $n$ is a power of $2$, then $G$ is a $2$-group and the claim of Theorem \ref{4.7} follows from  Theorem \ref{3.5}.  Thus, whenever convenient we suppose that $n \ge 3$. In particular, if $U \in \mathcal F (G)$ has type (b) with 
$x=n-1 \ge 2$, then, after replacing the basis $(e_1,e_2,e_3)$ by $(e_1+e_2+ne_3,e_2,e_3)$,
the sequence $U$ has type (a). Thus, in the sequel we will deal with type
\[
\text{(b*)} \  U = e_3^{v_3} (e_2+e_3)^{v_2} (e_1+xe_3) (e_1+e_2-xe_3), \ \text{where} \ v_3 \ge v_2\geq 1 \ \text{are odd},  \ v_2+v_3=2n, \ \text{and} \ x \in [2, n-2] \,.
\]
We also denote the subcase $v=0$ of type (d) in Proposition \ref{4.2} by type (d*) and, from now on, use type (d) for the remaining subcases $v \in [1, n-1]$.
These slight technical modifications allow us to show that the type of a minimal zero-sum sequence of maximal length is uniquely determined (Corollary \ref{4.4}.2). {
Whenever types (d*) and (d) occur together and no distinction between them is needed, we regard them as a single class and write (d*,d). This notation is consistent with Proposition \ref{4.2}: the subcases (d*) and (d) have disjoint parameter ranges, and their union is precisely type (d) in that proposition. Accordingly, we use (a)--(f) to represent the six classes (a), (b*), (c), (d*,d), (e), or (f).

Let $H=G[2]$. For every basis $\boldsymbol e=(e_1,e_2,e_3)$ of $G$ with $\ord(e_1)=\ord(e_2)=2$ and $\ord(e_3)=2n$, and every $V\in\mathcal F(G)$, we define
\[
A_{\boldsymbol e}(V)=\supp(V)\cap(e_3+H),\qquad B(V)=\supp(V)\cap H,
\]
and
\[
C_{\boldsymbol e}(V)=\supp(V)\cap\bigcup_{i\in[2,n-1]}(ie_3+H).
\]
These three sets form a partition of $\supp(V)$.

\begin{lemma}\label{4.3}
	Let $G=C_2^2\oplus C_{2n}$ with $n\ge3$, set $H=G[2]$, and let $U,U'\in\mathcal A_{\max}(G)$. Let $\boldsymbol e=(e_1,e_2,e_3)$ and $\boldsymbol f=(f_1,f_2,f_3)$ be bases of $G$ whose first two elements have order $2$ and whose third elements have order $2n$.
	
	\begin{enumerate}
		\item If $U$ belongs to one of the classes (a)--(f) with respect to $\boldsymbol e$, and also belongs to one of these classes with respect to $\boldsymbol f$, then
		\[
		e_3+H=f_3+H.
		\]
		\item If $U$ belongs to one of these classes with respect to $\boldsymbol e$, $U'$ belongs to one of these classes with respect to $\boldsymbol f$, and $T\mid\gcd(U,U')$ with $|T|=2n$, then
		\[
		e_3+H=f_3+H.
		\]
	\end{enumerate}
	Consequently, in either case,
	\[
	A_{\boldsymbol e}(V)=A_{\boldsymbol f}(V)
	\quad\text{and}\quad
	C_{\boldsymbol e}(V)=C_{\boldsymbol f}(V)
	\]
	for every $V\in\mathcal F(G)$. Moreover, $\mathsf h(U)\ge2$.
\end{lemma}

\begin{proof}
Let $W\in\mathcal A_{\max}(G)$ have one of the listed classes with respect to a basis $\boldsymbol q=(q_1,q_2,q_3)$. A close inspection of the forms in Proposition \ref{4.2}, together with the modifications above, shows the following properties.
	\begin{itemize}
		\item In types (a), (b*), (c), and (d*), the support of $W$ has at most four elements.
		\item In types (d) and (e), the support of $W$ has at most five elements.
		\item In type (f), the support of $W$ has at most six elements; moreover, only type (f) can have six support elements.
	\end{itemize}
	Every repeated term of $W$ belongs to $q_3+H$. More explicitly,
	\begin{itemize}
		\item in type (a), the possible repeated terms are $q_3$, $q_2+q_3$, and $q_1+q_3$;
		\item in type (b*), they are $q_3$ and $q_2+q_3$;
		\item in types (c), (d*), and (e), only $q_3$ can be repeated;
		\item in type (d), only $q_3$ and $q_2+q_3$ can be repeated;
		\item in type (f), the terms $q_1$ and $q_2$ occur once, and every other term belongs to $q_3+H$.
	\end{itemize}
	Since $|U|=2n+2>6\ge|\supp(U)|$, we have $\mathsf h(U)\ge2$.

For assertion 1, choose a term $g$ that is repeated in $U$. Applying the preceding observation to the two representations of $U$ gives
	\[
	g\in(e_3+H)\cap(f_3+H),
	\]
	and hence $e_3+H=f_3+H$.

	For assertion 2, suppose first that $\mathsf h(T)\ge2$, and choose $g$ with $\mathsf v_g(T)\ge2$. Then $g$ is repeated in both $U$ and $U'$, so the same observation again gives $g\in(e_3+H)\cap(f_3+H)$.

	It remains to consider the case that $\mathsf h(T)=1$. In this case $T$ is squarefree, and therefore
	\[
	2n=|\supp(T)|\le\min\{|\supp(U)|,|\supp(U')|\}\le6.
	\]
	Since $n\ge3$, it follows that $n=3$ and $|\supp(U)|=|\supp(U')|=6$. Thus both $U$ and $U'$ have type (f). A sequence of type (f) has only the two terms $q_1,q_2$ in $H$. Hence we may choose $g\in\supp(T)\setminus H$, and then again $g\in(e_3+H)\cap(f_3+H)$. This proves that  $e_3+H=f_3+H$ in all cases.

	Finally, this equality implies $ie_3+H=if_3+H$ for every integer $i$, and the assertions concerning $A$ and $C$ follow.
\end{proof}

By Lemma \ref{4.3}, the sets $\mathsf A_{\boldsymbol e} (V)$ and $C_{\boldsymbol e} (V)$ are independent of the basis $\boldsymbol e$. Thus, in every context covered by Lemma \ref{4.3}, we simply write $A(V)$, $B(V)$, and $C(V)$. Since Lemma \ref{4.3} is used extensively in this section, we will use this notation without further reference to the lemma. Proposition \ref{4.2} and Lemma \ref{4.3} now provide the data in Table \ref{BOX1}.

Every group automorphism $\varphi \colon G \to G$ maps a basis of $G$ onto a basis of $G$ and, conversely, any two bases $\boldsymbol e = (e_1, e_2, e_3)$ and $\boldsymbol f = (f_1,f_2,f_3)$, with $\ord (e_i)=\ord (f_i)=2$ for $i \in [1,2]$ and $\ord (e_3)=\ord (f_3)=2n$, give rise to an automorphism $\varphi \colon G \to G$ satisfying $\varphi (e_i)=f_i$ for $i \in [1,3]$. If a minimal zero-sum sequence $U \in \mathcal A_{\max} (G)$ is of some type (a) - (f) with respect to a basis $\boldsymbol e$, then $\varphi (U)$ is of the same type with respect to the basis $\varphi (\boldsymbol e)$ for every automorphism $\varphi \colon G \to G$.

	\begin{table}
		\centering
		\caption{~}
		\label{BOX1}
		
\begin{tabular}{|c|c|c|c|c|c|}
			\hline
			$Type$
			& $|A(U)|$ & $|B(U)|$ & $|C(U)| $ & $\sum_{\alpha \in C(U)} \alpha$ & $|\supp (U)|$   \\
			\hline
			(a)
			& $3$ & $0$ & $1$ & - & $4$   \\
			\hline
			(b*)
			& $2$ & $0$ & $2$ &   $\in H$ & $4$    \\
			\hline
			(c)
			& $1$ & $0$ & $3$ &  $= e_3$ & $4$    \\
			\hline
			(d*)
			& $1$ & $1$ & $2$ &  $\in e_3+H$ & $4$   \\
			\hline
			(d)
			& $2$ & $1$ & $2$ &  $\in e_3+H$ & $5$   \\
			\hline
			(e)
			& $1$ & $0$ & $4$ &  $= 2e_3$ & $5$   \\
			\hline
			(f)
			& $[2,4]$ & $2$ & $0$ & - & $[4,6]$ \\
			\hline
		\end{tabular}
		
\end{table}

\begin{corollary} \label{4.4}
Let $G = C_2 \oplus C_2 \oplus C_{2n}$ with $n\geq 3$ and let $U \in \mathcal A_{\max} (G)$. Suppose that $U$ has one of the types (a)--(f) with respect to a basis $(e_1, e_2, e_3)$ of $G$ with $\ord (e_1)=\ord (e_2)=2$ and $\ord (e_3) = 2n$. 
\begin{enumerate}
\item If $g_1,g_2\in G$ are distinct with $\mathsf{v}_{g_1}(U)\ge 2$ and $\mathsf{v}_{g_2}(U)\ge 2$, then $\ord(g_1-g_2)=2$. 

\item There is no basis $(e_1',e_2',e_3')$ of $G$ with $\ord(e_1')=\ord(e_2')=2$ and $\ord(e_3')=2n$ such that $U$ has a different type with respect to $(e_1',e_2',e_3')$. Furthermore, let $U'\in\mathcal A_{\max}(G)$ and suppose that there is a sequence $T\mid\gcd(U,U')$ with $|T|=2n$ and $\supp(T)=\supp(U)$. Then $U'$ has the same type as $U$, except possibly that $U$ has type (d*) and $U'$ has type (d).

\item $U\in \mathcal F (G\setminus 2G)$.

\item If $U$ has one of the types (a), (b*), (c), then $|\supp(U)|=4$ and $\{g+\langle e_3\rangle:g\in\supp(U)\}= G/\langle e_3\rangle$. 
If $U$ has type (d*) or (d), then $|\supp(U)|=4$ for $v=0$ and $|\supp(U)|=5$ for $v \in [1, n-1]$; in both cases, $\{g+\langle e_3\rangle\colon g\in\supp(U)\}=G/\langle e_3\rangle$.

\noindent
If $U$ has the type (e), then $|\supp(U)|=5$ and $\{g+\langle e_3\rangle:g\in\supp(U)\}= \{0,e_1,e_2\}+\langle e_3\rangle$.

\noindent
If $U$ has the type (f), then $|\supp(U)| \in [4, 6]$.

\end{enumerate}
\end{corollary}

\begin{proof}
1. Inspecting the types (a)--(f) in Proposition \ref{4.2}, we see that, for every $U$ and every element $g\mid U$ with  $\mathsf v_g(U)\ge 2$, we have $g\in e_3 +\langle e_1, e_2\rangle\subseteq A(U)$. (Then the third basis element $e_3$ can be chosen as any element of $e_3 +\langle e_1, e_2\rangle$.)

2. First suppose that $U$ has two representations, with respect to bases $\boldsymbol e$ and $\boldsymbol f$. By Lemma \ref{4.3}(1), the partitions $A(U),B(U),C(U)$ obtained from the two bases coincide. The corresponding entries in Table \ref{BOX1} distinguish the types, and hence the two representations have the same type.

For the furthermore statement, let $U'$ and $T$ be as in the assertion and assume to the contrary that $U$ and $U'$ have distinct types.  Lemma \ref{4.3} shows the relationship between  the partitions associated with $U$ and $U'$. Since $\supp(T)=\supp(U)\subseteq\supp(U')$, we have
\[
|A(U)|\le |A(U')|,\qquad |B(U)|\le |B(U')|,\qquad |C(U)|\le |C(U')|.
\]
We now use Table \ref{BOX1}. If $U$ has type (e), the inequality for $C$ is impossible. If $U$ has type (c), the only case not immediately excluded by the three inequalities is that $U'$ has type (e). Then $C(U)\subsetneq C(U')$ and
\[
\sum_{\alpha\in C(U')}\alpha-\sum_{\alpha\in C(U)}\alpha=e_3,
\]
whereas the unique element of $C(U')\setminus C(U)$ cannot belong to $e_3+H$, a contradiction. If $U$ has type (d*), the inequalities leave only type (d) for $U'$, which is precisely the stated exception. If $U$ has type (f), the inequality for $B$ is impossible, and if $U$ has type (d), the inequality for $A$ excludes type (d*), and a different type with at least one $B$-element can only be type (f), which is excluded by the inequality for $C$. If $U$ has type (a), the only possible larger value of $|A(U')|$ occurs for type (f), which is again excluded by the inequality for $C$. Finally, if $U$ has type (b*), the only case not immediately excluded is that $U'$ has type (d). But then $C(U)=C(U')$, while $\sum_{\alpha\in C(U)}\alpha\in H$ and $\sum_{\alpha\in C(U')}\alpha\in e_3+H$, a contradiction. 

3. and 4. These are immediate consequences of Proposition \ref{4.2}.
\end{proof}

Let $G=C_2^{r-1} \oplus C_{2n}$ with $n \ge 1$ and $r \ge 2$. We freely use Proposition \ref{2.2}.
Let $(e_1, ...,e_r)$ be a basis of $G$ with $\ord (e_1)=\ldots = 
\ord (e_{r-1})=2$ and $\ord (e_r)=2n$.
Let $H \subseteq G$ be a proper subgroup. If  $\langle 2e_r \rangle \subseteq H$, then
there exists a subgroup $K$ with $H \subseteq  K \subseteq G$ and $ (G:K)=2$.
If there is no $K$ with $H \subseteq  K \subseteq G$ and $ (G:K)=2$, then $H$ is generated by $r$ independent elements.

\begin{lemma}\label{4.5}
	Let $G = C_2 \oplus C_2 \oplus C_{2n}$ with $n\geq 3$, and let $T \in \mathcal F (G\setminus 2G)$ be zero-sum free. In each of the following cases we have $G^{\bullet} \setminus \Sigma (T) \subseteq \alpha + H$ for some subgroup $H \subseteq G$ with $(G \colon H)=2$ and $\alpha \in G \setminus H$.
	\begin{enumerate}
		\item[(a)] $G^{\bullet} \setminus \Sigma (T) = \emptyset$.
		
		\item[(b)] $G^{\bullet} \setminus \Sigma (T) = \{g\}$ with $g\notin 2G$.
		
		\item[(c)] $G^{\bullet}\setminus \Sigma (T)=\{g,h\}$ and $g,h\notin 2G$.
		
		\item[(d)] $T \in \mathcal F (K\setminus 2G)$ for some subgroup $K \subseteq G$ with $(G \colon K)=2$ and $|T|=2n$.

        \item[(e)] $T$ has length $|T|=2n$ and divides  some $U\in\mathcal A_{\max}(G)$ of type (a), (b*), (c), or (e).
	\end{enumerate}
\end{lemma}

\begin{proof}
	(a) Obvious.
	
	(b) Since $2G$ is the intersection of all subgroups of $G$ with index $2$, there exists $K \subseteq G$ with $(G : K)=2$ and with $g\notin K$, whence $\{g\}\subseteq g+K$.
	
	(c) If $g\in \langle h-g \rangle$, say $g=k(h-g)$, then $h=(k+1)(h-g)\in\langle h-g\rangle$ for some $k\in\N$. Hence either $g\in2G$ or $h\in2G$, a contradiction. Thus we obtain that $g,h\notin\langle h-g\rangle$.
	
	Let $(e_1, e_2, e_3)$ be a basis of $G$ with $\ord (e_1)=\ord (e_2)=2$ and $\ord (e_3) = 2n$. Let $g=x_1e_1+x_2e_2+x_3e_3$ and $h=y_1e_1+y_2e_2+y_3e_3$, where $x_1,x_2,y_1,y_2\in \{0,1\}$ and $x_3,y_3\in [0,2n-1]$. We choose $x_3',y_3'\in\{0,1\}$ such that $x_3'\equiv x_3\mod 2$ and $y_3'\equiv y_3\mod 2$.
	
	Since $g,h\notin 2G$, both $(x_1,x_2,x_3')\ne (0,0,0)$ and $(y_1,y_2,y_3')\ne (0,0,0)$. Let $g'=x_1e_1+x_2e_2+x_3'e_3$ and $h'=y_1e_1+y_2e_2+y_3'e_3$. Then $g-h\in \langle g'-h',2e_3 \rangle$. We distinguish three cases.
	
	If $x_3'=y_3'=1$, then we set $H= \langle e_1,e_2,2e_3 \rangle$, thus $g',h'\in e_3+ H$.
	
	If $x_3'=0$ (or $y_3'=0$), and $h'=e_3$ (or $g'=e_3$), then $(x_1,x_2)\ne (0,0)$. We discuss one case in detail, the other ones run along the same lines. Suppose that $g'=e_1$. We set $H= \langle e_2,e_1+e_3 \rangle$, thus $g',h'\in g'+ H$.
	
	So we may assume $g',h'\ne e_3$. Thus $(x_1,x_2)\ne (0,0)$ and $(y_1,y_2)\ne (0,0)$. If $g'- h'\notin \langle e_3 \rangle$, then $\{g',h'\}= g'+\langle h'-g' \rangle$, we set $H=\langle h'-g',e_3 \rangle$. If $g'-h'\in \langle e_3 \rangle$, choose $f\in \langle e_1,e_2 \rangle \setminus \{0,x_1e_1+x_2e_2\}$, and set $H=\langle f-g',e_3 \rangle$. Thus $g',h'\in g'+ H$ for both cases.
	
	In each case, we have $g',h'\in g'+ H$ and $g-g'\in \langle 2e_3 \rangle\subseteq H$. So $g,h\in g'+ H$ with $g\notin H$, and $(G:H)=2$.
	
	(d) If $n$ is even and $K\cong C_2 \oplus C_2 \oplus C_{n}$, then $n\ge 2$ and $\mathsf{D}(C_2 \oplus C_2 \oplus C_{n})=n+2\le 2n =|T|$, contradicting the fact that $T$ is zero-sum free.
	
	Thus, it follows that $K\cong C_2 \oplus C_{2n}$. Since $|T|=\mathsf{d}(K)=2n$ and $T$ is zero-sum free, we obtain that $K^{\bullet}= \Sigma (T)$. Thus $G^{\bullet}\setminus \Sigma (T)=\alpha+K$ for some $\alpha \in G \setminus K$.
	
\smallskip
(e) Suppose that $|T|=2n$ and that $T\mid U$ for some $U\in\mathcal A_{\max}(G)$ of type (a), (b*), (c), or (e). We distinguish two cases.
	
	\smallskip
	\noindent
	CASE 1: $U$ has one of the types (a), (b*), or (c).
	
By Corollary \ref{4.4}.4, $|\supp(U)|=4$. Since $|T|=2n=|U|-2$, we have $|\supp(T)|\ge |\supp(U)|-2=2$. We distinguish four subcases.
	\begin{itemize}
		\item[(i)] $|\supp(T)|=2$. Then $T \in \mathcal F (K\setminus 2G)$ for some subgroup $K \subseteq G$ with $(G \colon K)=2$, and we are done.
		
		\item[(ii)] $|\supp(T)|=3$ and $\mathsf{v}_g(T)=1$ for some $g\in G$. Thus $|\supp(Tg^{-1})|=2$, and $Tg^{-1}\in\mathcal F(K\setminus 2G)$ for $K=\langle\supp(Tg^{-1})\rangle$ with $(G\colon K)=2$. If $g\in K$, then $T\in\mathcal F(K\setminus 2G)$ and the assertion follows from the first part of (d). Hence we may assume that $g\notin K$, so $T_K=Tg^{-1}$. By Proposition \ref{4.2}, either $T_K=e_3^{v_3}(a+e_3)^{v_2}$, where $v_2,v_3$ are positive and $v_2+v_3=2n-1$, or $T_K=e_3^{2n-2}(a+xe_3)$, where $x\in[2,n-1]$; here $a\in\langle e_1,e_2\rangle$.

		Suppose that $T_K=e_3^{v_3}(a+e_3)^{v_2}$. If $v_2$ is odd, then $K^{\bullet}\setminus\Sigma(T_K)=\{-e_3,a\}$, and hence $G^{\bullet}\setminus\Sigma(T)=\{-e_3,a,g-e_3,g+a\}$. Let $H=\langle a+e_3,g,2e_3\rangle$. Since $K=\langle a,e_3\rangle$ and $g\notin K$, the images of $a+e_3$ and $g$ modulo $2G=\langle2e_3\rangle$ are linearly independent, and their span does not contain $e_3+2G$. Thus $(G\colon H)=2$ and $e_3\notin H$, and therefore $G^{\bullet}\setminus\Sigma(T)\subseteq-e_3+H$. If $v_2$ is even, then $K^{\bullet}\setminus\Sigma(T_K)=\{a,a-e_3\}$, and hence $G^{\bullet}\setminus\Sigma(T)=\{a,a-e_3,g+a,g+a-e_3\}$. Let $H=\langle g,e_3\rangle$. Since $g\notin K=\langle a,e_3\rangle$, we have $(G\colon H)=2$ and $a\notin H$, and therefore $G^{\bullet}\setminus\Sigma(T)\subseteq a+H$.
		
		If $T_K=e_3^{2n-2}(a+xe_3)$, then $K^{\bullet}\setminus\Sigma(T_K)=\{-e_3,a+(x-1)e_3\}$. We have $G^{\bullet}\setminus\Sigma(T)=\{-e_3,a+(x-1)e_3,g-e_3,g+a+(x-1)e_3\}$. Let $H=\langle a+xe_3,g,2e_3\rangle$. Since $K=\langle a,e_3\rangle$ and $g\notin K$, the images of $a+xe_3$ and $g$ modulo $2G=\langle2e_3\rangle$ are linearly independent, and their span does not contain $e_3+2G$. Thus $(G\colon H)=2$ and $e_3\notin H$, and therefore $G^{\bullet}\setminus\Sigma(T)\subseteq-e_3+H$.
		
		\item[(iii)] $|\supp(T)|=3$ and $\mathsf{v}_g(T)\ge 2$ for all $g\mid T$. Since $U$ has one of the types (a), (b*), (c) in Proposition \ref{4.2}, it must have type (a). Thus $T= e_3^{v_3} (e_2+e_3)^{v_2} (e_1+e_3)^{v_1}\mid U$, where $v_3 \ge v_2 \ge v_1\geq 2$ are not all even with $v_3+v_2+v_1=2n$.
		
		Exactly one of $v_1,v_2,v_3$ is even; after permuting the three repeated terms and changing the basis accordingly, we may assume that $v_2$ is even. We discuss this case in detail; the other cases run along the same lines.
		
		If $v_2$ is even, then $G^{\bullet}\setminus \Sigma (T)=\{-e_2-e_3,-e_1-e_2+e_3, -e_2,-e_1-e_2\}\subseteq e_2+ \langle e_1,e_3\rangle$.
		
		\item[(iv)] $|\supp(T)|=4$. Then $\supp(T)=\supp(U)$. Let $U'\in\mathcal A_{\max}(G)$ be arbitrary with $T\mid U'$. By Corollary \ref{4.4}.2, $U'$ has the same type as $U$. Hence $\supp(U')=\supp(T)=\supp(U)$.
			
			Choose $g_0\in\supp(T)$ with $\mathsf v_{g_0}(T)\ge2$. If $g\in\supp(U'T^{-1})$, then both $g_0$ and $g$ are repeated in $U'$. Thus, if $g\ne g_0$, Corollary \ref{4.4}.1 gives $\ord(g-g_0)=2$. Inspecting the forms (a), (b*), and (c) in Proposition \ref{4.2}, we obtain $\supp(U'T^{-1})\subseteq e_3+\langle e_1,e_2\rangle$. Since $U'$ was arbitrary, Proposition \ref{2.6} yields $G^{\bullet}\setminus\Sigma(T)\subseteq -e_3+\langle e_1,e_2,2e_3\rangle$, and $(G:\langle e_1,e_2,2e_3\rangle)=2$.
	\end{itemize}
	
	\smallskip
	\noindent
	CASE 2: $U$ has the type (e).
	
	Then $|T|=|U|-2$ implies that $|\supp(T)|\ge |\supp(U)|-2=3$. Since $n\ge 3$, we infer that $\mathsf{v}_{e_3}(T)\ge 2n-4\ge 2$ and $A(T)=\{e_3\}$. We distinguish three subcases.
	
	\begin{itemize}
		\item[(i)] $|\supp(T)|=5$. Then $|C(T)|=4$. We will use Proposition \ref{2.6}. For this we choose some  $U'\in\mathcal A_{\max}(G)$  with $T\mid U'$. By Table \ref{BOX1}, $U'$ has type (e), and hence $\supp(U')=\supp(T)=\supp(U)$. Since $\mathsf v_{e_3}(T)\ge2$ and only the distinguished third basis element can be repeated in a sequence of type (e), we obtain $U'=U$ and $U'T^{-1}=e_3^2$. Since $U'$ was arbitrary, Proposition \ref{2.6} gives $G^{\bullet}\setminus\Sigma(T)=\{-e_3\}$, and we are done by Case (b).
		
		\item[(ii)] $|\supp(T)|=4$. Then $|C(T)|=3$  and $\mathsf v_{e_3}(T)=2n-3$. Let $U'\in\mathcal A_{\max}(G)$ be arbitrary with $T\mid U'$. By Table \ref{BOX1}, every such $U'$ has type (c) or (e).
			
			If there exists such a $U'$ of type (c), then the assertion follows from CASE 1 applied to $U'$. Hence we may assume that every $U'\in\mathcal A_{\max}(G)$ with $T\mid U'$ has type (e). For every such $U'$, the repeated term is $e_3$, and the unique element of $C(U')\setminus C(T)$ is determined by $\sum_{\alpha\in C(U')}\alpha=2e_3$. Consequently, $U'=U$. Proposition \ref{2.6} now gives $G^{\bullet}\setminus\Sigma(T)=\{-e_3,\sigma(Te_3)\}$, and we are done by Case (c).
		
		\item[(iii)] $|\supp(T)|=3$. Therefore, we obtain that $\mathsf{v}_{e_3}(T)=2n-2$. We set $T= e_3^{2n-2} gh$, where $C(T)=\{g,h\}\subseteq C(U)$.

		If $g+h=e_3$, then $H:=\langle\supp(T)\rangle=\langle g,e_3\rangle \subseteq G$. By the form of a sequence of type (e), $H=\langle e_1,e_3\rangle$ or $H=\langle e_2,e_3\rangle$, and hence $\mathsf d(H)=2n$. Since $(G:\langle e_3\rangle)=4$, we have $(G:H)=2$. Since $T\in\mathcal F(H)$ is zero-sum free and $|T|=2n=\mathsf d(H)$, we have $\Sigma(T)=H^{\bullet}$. Therefore $G^{\bullet}\setminus\Sigma(T)=G\setminus H=\alpha+H$ for any $\alpha\in G\setminus H$.
		
		Suppose now that $g+h\ne e_3$. Then $G^{\bullet}\setminus\Sigma(T)=\{-e_3,-e_3+g,-e_3+h,-e_3+g+h\}$. By the form of a sequence of type (e), the terms $g$ and $h$ belong to distinct nonzero cosets modulo $\langle e_3\rangle$. Hence their images modulo $2G=\langle2e_3\rangle$ are linearly independent and their span does not contain $e_3+2G$. Let $H:=\langle g,h,2e_3\rangle$. Then $(G\colon H)=2$, $e_3\notin H$, and  $G^{\bullet}\setminus\Sigma(T)\subseteq-e_3+H$. \qedhere
\end{itemize}
\end{proof}

\begin{lemma} \label{4.6}
		Let $G=C_2 \oplus C_2 \oplus C_{2n}$ with $n \ge 1$ and let $S\in \mathcal F (G)$ be zero-sum free with $|S|=\mathsf d (G)-1$. Suppose there exists a subsequence $T \mid S$ with $|T|\ge |S|-2$ and 
	a basis $(e_1,e_2,e_3)$ of $G$ with $\ord (e_1)=\ord (e_2)=2$, and with $\ord (e_3)=2n$ such that $(-e_3)+T\in \mathcal{F}(\langle e_1,e_2\rangle)$. Then there is a subgroup $K \subsetneq G$ with $(G : K) = 2$  and some $x \in G \setminus K$ such that
	\begin{align*}
	G^{\bullet} \setminus \Sigma (S) \subseteq x + K \,.
	\end{align*}
\end{lemma}

\begin{proof}
If $n\in\{1,2\}$, then $G$ is a $2$-group and the assertion follows from Theorem \ref{3.5}. Hence suppose that $n\ge 3$.

By Proposition \ref{4.2}, we have $\mathsf d (G) = 2n+1$, whence $|S|=2n$.
Since $-e_3+T\in\mathcal F(\langle e_1,e_2\rangle)$, every term of $T$ belongs to $e_3+\langle e_1,e_2\rangle$, and hence $T\mid S_{e_3+\langle e_1,e_2\rangle}$. Replacing $T$ by $S_{e_3+\langle e_1,e_2\rangle}$, we may therefore assume that $T$ contains all terms of $S$ belonging to this coset. The new $T$ still satisfies $-e_3+T\in\mathcal F(\langle e_1,e_2\rangle)$, and its length can only increase; thus the inequality $|T|\ge |S|-2$ is preserved.

	Let $R \in \mathcal F (G)$ be the largest subsequence of $T$ that $\mathsf v_g (R)$ is even for all $g \in G$. We set  
\[
|ST^{-1}|=s\le 2 \,,  \quad |TR^{-1}|=t \,, \quad \text{and} \quad   |R|=2m \,.
\] 
Since $-e_3+T \in \mathcal F ( \langle e_1, e_2 \rangle)$, it follows that one of the following four conditions holds:
\[
e_3 \t T, \ (e_1+e_3) \t T, \ (e_2+e_3) \t T , \quad \text{or} \quad (e_1+e_2+e_3) \t T \,.
\]
Set $f_1=e_1$ and $f_2=e_2$. If $t > 0$, we choose some $f_3 \in \supp ( TR^{-1})$ and if $t=0$, we choose some $f_3\in\supp(T)$. Since $f_3\in e_3+\langle e_1,e_2\rangle$, in both cases the triple $(f_1, f_2, f_3)$ is a basis of $G$ and  $-f_3+T\in\mathcal F(\langle f_1,f_2\rangle)$. After renaming the basis if necessary, we may suppose that $e_3\mid TR^{-1}$ when $t>0$, and that $e_3^2\mid R$ when $t=0$.

Note that  $2g=2e_3$ for every $g\mid T$. 
We set
\begin{equation}\label{def_R}
R = \prod_{i=1+t}^{m+t} (a_ie_1+b_ie_2 + e_3)^2 \qquad \text{and} \qquad T= g_1 \cdot \ldots \cdot  g_t R \,, 
\end{equation}
where
\[
a_i, b_i \in [0,1] \ \text{for all } i \in [1+t,m+t] , \ \text{and} \ g_i=a_ie_1+b_ie_2 + e_3 \ \text{ with } \ a_i, b_i \in [0,1] \ \text{for all} \ i \in [1,t] \,.
\]
Thus we have $e_3\mid TR^{-1}$ if $t\ne 0$, and let $g_1=e_3$ for this case. Since $\mathsf v_g (TR^{-1})\le 1$ for every $g\in G$ and $\supp (TR^{-1}) \subseteq e_3+ \langle e_1,e_2\rangle$, it follows that  
\[
t=|TR^{-1}|\le 4 \quad \text{and} \quad  |R|\ge |T|-4\ge |S|-6 \,.
\] 
We have
\[
|R| = |T|-t \le |S|-t \le \mathsf d (G)-1 -t \le 2n \,.
\]
Since $R$ is zero-sum free and $2g=2e_3$ for all $g \in \supp (R)$, it follows that $|R| < 2n = \exp (G)$, whence
\begin{equation}
m\in[n-3,n-1] \,.
\end{equation}
  If $s=1$, we set $S = T \alpha$ and if $s=2$, we set $S=T \alpha \beta$, where $\alpha = \alpha_1 e_1 + \alpha_2 e_2 + \alpha_3 e_3$ and $\beta = \beta_1 e_1 + \beta_2 e_2 + \beta_3 e_3$, with $\alpha_1, \alpha_2, \beta_1, \beta_2 \in [0,1]$ and $\alpha_3, \beta_3 \in \{0\}\cup [2, 2n-1]$. Indeed, the coefficient $1$ is excluded by the present choice $T=S_{e_3+\langle e_1,e_2\rangle}$.

  First suppose that  $t=0$. Then $|R|= |T|\ge |S|-2$  and $m= n-1$. So $|ST^{-1}|=s=2$ and 
	  we obtain $S = R \alpha \beta$. Based on the property of our basis, $e_3^2\mid R$ and then $\{e_3,...,(2n-2)e_3\}=\Sigma(e_3^2 (2e_3)^{n-2}) =\Sigma(e_3^{2n-2})\subseteq \Sigma(R)$. We set $R_0=e_3^{2n-2}$ and observe that $\Sigma(R_0)\subseteq\Sigma(R)$. So $\alpha, \beta\notin \langle e_3\rangle$ and $\alpha +\beta\notin \langle e_3\rangle \setminus \{e_3\}$.  Since $2G^{\bullet} \subseteq \Sigma((2e_3)^{n-1}) \subseteq \Sigma(R)$, we have $\alpha +\beta\notin 2G$ and $\alpha \ne \beta$. We distinguish two cases.
  \begin{itemize}
	  	\item[(i)] $\alpha+\beta=e_3\in\langle e_3\rangle$. Then $\langle\alpha,e_3\rangle=\langle\supp(R_0\alpha\beta)\rangle\cong C_2\oplus C_{2n}$ and $(G:\langle\alpha,e_3\rangle)=2$. The sequence $R_0\alpha\beta$ is zero-sum free, has length $\mathsf d(\langle\alpha,e_3\rangle)$, and satisfies $\Sigma(R_0\alpha\beta)\subseteq\Sigma(S)$. Hence $\langle\alpha,e_3\rangle^{\bullet}=\Sigma(R_0\alpha\beta)\subseteq\Sigma(S)$, and $G^{\bullet}\setminus\Sigma(S)\subseteq x+\langle\alpha,e_3\rangle$ for any $x\in G\setminus\langle\alpha,e_3\rangle$.
  	
  	\item[(ii)] $\alpha +\beta\notin \langle e_3\rangle $. If $e_3=u\alpha+v\beta$ where $u$ and $v$ are not both even, then $u\alpha+v\beta\in \langle e_3\rangle$. Since $2\alpha,2\beta\in \langle e_3\rangle$, we have $\alpha$ or $\beta$ or $\alpha+\beta\in \langle e_3\rangle$, a contradiction. So $e_3\notin \langle \alpha ,\beta\rangle$. Then
	  A direct calculation with the auxiliary sequence gives
		$G^{\bullet}\setminus\Sigma(S)\subseteq G^{\bullet}\setminus\Sigma(R_0\alpha\beta)\subseteq-e_3+\{0,\alpha,\beta,\alpha+\beta\}\subseteq-e_3+\langle\alpha,\beta,2e_3\rangle$, where $(G:\langle\alpha,\beta,2e_3\rangle)=2$.
  \end{itemize}

 \smallskip
From now on we have
\begin{equation}
m\in[n-3,n-1] \quad \text{and} \quad t \in [1, 4] \,. 
\end{equation}
Then $R'=Rg_1$ and  $ \{e_3,...,(2m+1)e_3\}=\Sigma(e_3^{2m+1})=\Sigma(e_3 (2e_3)^{m}) \subseteq \Sigma(R')$. 

Suppose that $s=0$. Then $t=2n-2m\in \{2,4\}$. We distinguish two cases.
\begin{itemize}
\item[(i)] $t=2$. We set $S = R g_1 g_2=R'g_2$, whence $\langle e_3\rangle^{\bullet}\subseteq \Sigma (R')$. 
	Since $S$ is zero-sum free, we infer that $g_2\notin \langle e_3\rangle$. Thus, 
	$\langle g_2,e_3 \rangle^{\bullet}\subseteq \Sigma (e_3^{2n-1} g_2)\subseteq \Sigma (S)$ with $(G: \langle g_2, e_3 \rangle)=2$, whence  $G^{\bullet} \setminus \Sigma (S) \subseteq x + \langle g_2, e_3 \rangle$ for some $x \in G \setminus \langle g_2, e_3 \rangle$.
	
\item[(ii)] $t=4$. We have $S=Rg_1g_2g_3g_4$, where
$g_1g_2g_3g_4=e_3(e_1+e_3)(e_2+e_3)(e_1+e_2+e_3)$. Thus $\sigma(S)=m(2e_3)+4e_3=0$, a contradiction.
\end{itemize}

\smallskip
From now on we have 
\begin{equation}
m\in[n-3,n-1], \quad t \in [1, 4], \quad \text{and} \quad s \in [1, 2] \,. 
\end{equation}
Recall $ \alpha$  (and $\beta)$ as above, then $\alpha_3(, \beta_3) \in \{0\}\cup [2, 2n-1]$.

\smallskip
Suppose that $\alpha_3=0$; the case $\beta_3=0$ is symmetric. Let $\varphi \colon G\to \overline G:= G/\langle\alpha\rangle$ be the canonical epimorphism and put $\overline S=\varphi(S\alpha^{-1})$. Then $\overline G\cong C_2\oplus C_{2n}$ and $\overline S$ is zero-sum free. Indeed, if a nontrivial subsequence of $S\alpha^{-1}$ had sum in $\langle\alpha\rangle$, then either this subsequence or its product with $\alpha$ would be a nontrivial zero-sum subsequence of $S$.

Moreover, if $y\in G^{\bullet}\setminus\Sigma(S)$, then $\varphi(y)\in\overline G^{\bullet}\setminus\Sigma(\overline S)$. Indeed, $y\in\langle\alpha\rangle$ would force $y=\alpha\in\Sigma(S)$; and if $\varphi(y)\in\Sigma(\overline S)$, then a subsequence of $S\alpha^{-1}$ would have sum $y$ or $y-\alpha$, again forcing $y\in\Sigma(S)$. Since $|\overline S|=2n-1=\nu_2(\overline G)$, Lemma \ref{4.1} yields an index-$2$ subgroup $\overline K\le\overline G$ and some $\overline x\notin\overline K$ such that
\[
\overline G^{\bullet}\setminus\Sigma(\overline S)\subseteq\overline x+\overline K.
\]
Set $K=\varphi^{-1}(\overline K)$ and choose $x\in\varphi^{-1}(\overline x)$. Then $(G:K)=2$, $x\notin K$, and $G^{\bullet}\setminus\Sigma(S)\subseteq x+K$.

\smallskip
From now on we have 
\begin{equation}
m\in[n-3,n-1], \quad t \in [1, 4], \quad s \in [1, 2] , \quad \text{and} \quad  \alpha_3, \beta_3 \in [2,2n-1] \,. 
\end{equation}

	\smallskip
	\noindent
	CASE 1: $m=n-1$. 
	%
	%

	
	We set $S = R g_1 h=R'h$, then $\langle e_3\rangle^{\bullet}\subseteq \Sigma (R')$. 
	Since $S$ is a zero-sum free sequence,  we have $h\notin \langle e_3\rangle$. Thus,
	$\langle h,e_3 \rangle^{\bullet}\subseteq \Sigma (e_3^{2n-1} h)\subseteq \Sigma (S)$ with $(G: \langle h, e_3 \rangle)=2$, whence $G^{\bullet} \setminus \Sigma (S) \subseteq x + \langle h, e_3 \rangle$ for some $x \in G \setminus \langle h, e_3 \rangle$.

	%
	%
	%
	%
	%
	%
	%
	%
	%
	
	%
	%
	%
	%
	%
	%
	
	\smallskip
	\noindent
	CASE 2:  $m=n-3$.
	
	Then $|R|=2n-6$, $t=4$,   $s=2$, and $S=Rg_1g_2g_3g_4\alpha \beta$ with $g_2g_3g_4=(e_1+e_3)(e_2+e_3)(e_1+e_2+e_3)$. We distinguish two cases.

Suppose that  $n=3$. Then $R=1$. Since $g_1=e_3$ and $g_2g_3g_4$ has the form given above, a direct calculation gives
	$\Sigma^*(g_1g_2g_3g_4)\cap\langle e_3\rangle
	=\{0,e_3,3e_3,4e_3\},$
	whereas 
	$\Sigma^*(g_1g_2g_3g_4)\setminus\langle e_3\rangle
	=
	\cup_{i=1}^3
	\{e_1+ie_3,e_2+ie_3,e_1+e_2+ie_3\}.$
	Since $S$ is zero-sum free, $-\alpha\notin\Sigma(g_1g_2g_3g_4)$. Hence $\alpha_3=4$ if $\alpha_1=\alpha_2=0$, while $\alpha_3=2$ otherwise. The same conclusion holds for $\beta$. Moreover,
	$-\sigma(\alpha\beta)\notin\Sigma^*(g_1g_2g_3g_4).$
	It follows, up to interchanging $\alpha$ and $\beta$, that precisely the following two possibilities remain:
\begin{itemize}
\item $\alpha_1=\alpha_2=0$ and $(\beta_1,\beta_2)\ne(0,0)$, in which case $\{\alpha,\beta\}=\{4e_3,\beta_1e_1+\beta_2e_2+2e_3\}$;
\item $(\alpha_1,\alpha_2)\ne(0,0)$ and $\beta_1=\beta_2=0$, in which case $\{\alpha,\beta\}=\{\alpha_1e_1+\alpha_2e_2+2e_3,4e_3\}$;
\item $(\alpha_1,\alpha_2)=(\beta_1,\beta_2)\ne(0,0)$, in which case $\alpha=\beta=\alpha_1e_1+\alpha_2e_2+2e_3$.
\end{itemize}
All other possibilities yield a zero-sum subsequence by $-\sigma(\alpha\beta)\in\Sigma^*(g_1g_2g_3g_4)$. 
For the first and second possibility, a direct calculation gives $\Sigma(S)=G^{\bullet}$, and hence the assertion is immediate. For the last possibility, a direct calculation gives $G^{\bullet}\setminus\Sigma(S)=\{\alpha_1e_1+\alpha_2e_2+4e_3\}$. This unique missing element does not belong to $2G$, and hence the assertion follows from Lemma \ref{4.5}(b). 

Now we  suppose that $n\ge4$. 
If a nontrivial subsequence of $g_2g_3g_4\alpha\beta$ has sum $ce_3$, then $c\in[1,4]$: the value $c=0$ contradicts zero-sum freeness directly, while every $c\in[5,2n-1]$ can be cancelled by an element of $\{e_3,\ldots,(2n-5)e_3\}\subseteq\Sigma(R')$.

If $\alpha_1=\alpha_2=0$, then $\alpha_3\in[2,4]$. Otherwise, the unique term among $g_2,g_3,g_4$ having the same $\langle e_1,e_2\rangle$-component as $\alpha$, together with $\alpha$, forms a subsequence whose sum lies in $\langle e_3\rangle$; the preceding restriction therefore gives $\alpha_3\in[2,3]$. The analogous bounds hold for $\beta_3$. If $(\alpha_1,\alpha_2)\ne(\beta_1,\beta_2)$, then $\alpha\beta$ together with the unique term among $g_2,g_3,g_4$ whose $\langle e_1,e_2\rangle$-component equals $(\alpha_1+\beta_1)e_1+(\alpha_2+\beta_2)e_2$ has sum in $\langle e_3\rangle$. Its $e_3$-coefficient lies in $[5,8]$ when one of $(\alpha_1,\alpha_2)$ and $(\beta_1,\beta_2)$ is $(0,0)$, and in $[5,7]$ otherwise, contradicting the preceding restriction. Hence $(\alpha_1,\alpha_2)=(\beta_1,\beta_2)$. Applying the same restriction to $\alpha\beta$ gives $\alpha_3=\beta_3=2$.
Thus $\sigma(g_2g_3g_4\alpha\beta)=7e_3$. Since $(2n-7)e_3\in\Sigma(R')$, we obtain $0\in\Sigma(S)$, a contradiction. Therefore CASE 2 cannot occur when $n\ge4$.
	
\smallskip
\noindent
CASE 3:  $m=n-2$. 
	
Then $|R|=2n-4$, $s\in[1,2]$, $t=4-s$ and $\Sigma(R')\supseteq \{e_3,...,(2n-3)e_3\}$, whence  
\[
\langle e_3\rangle \cap \Sigma (SR'^{-1}) \subseteq \{ e_3,2e_3\} \,.
\]
We use the following {\it Argument ($*$)} in all the subcases below. Let $h_1,\ldots,h_\ell\in\Sigma(SR'^{-1})$, and choose $W_1\mid SR'^{-1}$ with $\sigma(W_1)=h_1$. Since $\{e_3,\ldots,(2n-3)e_3\}\subseteq\Sigma(R')$, we have $\{h_1,h_1+e_3,\ldots,h_1+(2n-3)e_3\}\subseteq\sigma(W_1)+\Sigma^*(R')\subseteq\Sigma(S)$. The only two remaining elements of the coset $h_1+\langle e_3\rangle$ are $h_1+(2n-2)e_3=h_1-2e_3$ and $h_1+(2n-1)e_3=h_1-e_3$. Hence $(h_1+\langle e_3\rangle)\setminus\Sigma(S)\subseteq\{h_1-2e_3,h_1-e_3\}$. 
If $h_i-h_1\in\langle e_3\rangle$ for some $i\in[2,\ell]$, then $h_i$ belongs to the same coset as $h_1$, and the same argument shows that the missing elements of this coset are also contained in $\{h_i-2e_3,h_i-e_3\}$. In each of the four cosets $\langle e_3\rangle$, $e_1+\langle e_3\rangle$, $e_2+\langle e_3\rangle$, and $e_1+e_2+\langle e_3\rangle$ separately, we intersect the two-element sets obtained from all $h_i$ belonging to that coset. The resulting four missing-set containments are used in the following subcases.

	\smallskip
	\noindent
	CASE 3.1:  $s=2$. 

Then $S=R'g_2\alpha \beta$ with $g_2 \in \{ e_1+e_3, e_2+e_3, e_1+e_2+e_3\}$. We  discuss one case in detail, the other ones run along the same lines. Suppose that $g_2=e_1+e_3$.   Then $\langle e_1,e_3\rangle^{\bullet} \setminus \Sigma (R'g_2)\subseteq \{-e_3,e_1,-2e_3,e_1-e_3\}$. 
	
\smallskip
\noindent
CASE 3.1.1: $\alpha_2=0$ or $ \beta_2=0$, say $\alpha_2=0$.

We distinguish two cases.
\begin{itemize}
\item[(i)] $\alpha_1=\alpha_2= 0$, the preceding containment forces $\alpha_3=2$, and $\alpha=2e_3 \in \langle e_3\rangle $. So $\Sigma(Rg_1\alpha)\supseteq \Sigma(e_3^{2n-3}(2e_3))= \langle e_3\rangle^{\bullet}$, and hence $\Sigma(Rg_1g_2\alpha)\supseteq \langle e_1,e_3\rangle^{\bullet}$ and $G^{\bullet} \setminus \Sigma (S) \subseteq e_2 + \langle e_1,e_3\rangle$  with $ (G:\langle e_1,e_3\rangle)=2$ and we are done.
		
\item[(ii)] $\alpha_1=1 $ and $ \alpha_2=0$. Thus $\sigma(g_2\alpha)=(1+\alpha_3)e_3\in \{e_3,2e_3\}$ and $\alpha_3\le 1$, a contradiction to $2\le \alpha_3$.
\end{itemize}

	\smallskip
	\noindent
	CASE 3.1.2: $\alpha_1= \beta_1$ and $ \alpha_2= \beta_2=1$. 

Then $\sigma(\alpha \beta)=(\alpha_3+\beta_3)e_3\in \{e_3,2e_3\}$. We distinguish two cases.
\begin{itemize}
\item[(i)] $\sigma(\alpha \beta)=2e_3$. Thus $\Sigma(Rg_1\alpha \beta)\supseteq \Sigma(e_3^{2n-3}(2e_3))= \langle e_3\rangle^{\bullet}$ and $\Sigma (S) \supseteq  \langle g_2,e_3\rangle^{\bullet}$  with $ (G:\langle g_2,e_3\rangle)=2$. So $G^{\bullet} \setminus \Sigma (S) \subseteq x + \langle g_2,e_3\rangle$ for any $x\in G\setminus \langle g_2,e_3\rangle$ and we are done. 
		
\item[(ii)] $\sigma(\alpha \beta)=e_3$.  Without loss of generality, let $\alpha_1= \beta_1=0$. Then $\alpha=e_2+\alpha_3 e_3$ and $\beta =e_2+(2n+1-\alpha_3)e_3$. 

{\it Argument ($*$)} implies that, for every admissible value of $\alpha_3$,
\[
G^{\bullet}\setminus\Sigma(S)\subseteq\{-e_3,e_1,e_2+(n-1)e_3,e_1+e_2+ne_3\}
\subseteq-e_3+H,
\]
where $H=\langle e_1+e_3,e_2+ne_3,2e_3\rangle$. Since $(G:H)=2$ and $e_3\notin H$, we are done.
\end{itemize}
	
\smallskip
\noindent
CASE 3.1.3:  $\alpha_1\ne \beta_1$ and $ \alpha_2= \beta_2=1$,  say $\alpha_1= 0$ and $\beta_1=1$.

Then $\sigma(g_2\alpha \beta)=(1+\alpha_3+\beta_3)e_3\in \{e_3,2e_3\}$. We distinguish two cases.
\begin{itemize}
\item[(i)] $\sigma(g_2\alpha \beta)=e_3$.  Then $\sigma(\alpha \beta)=e_1$,  $\alpha=e_2+\alpha_3e_3$, and $\beta =e_1+e_2+(2n-\alpha_3)e_3$. 

For every admissible value of $\alpha_3$, the {\it Argument ($*$)} implies
\[
G^{\bullet}\setminus\Sigma(S)
\subseteq\{-e_3,e_1-e_3,e_2+(n-1)e_3,e_1+e_2+(n-1)e_3\}
\subseteq-e_3+\langle e_1,e_2+ne_3,2e_3\rangle.
\]
The displayed subgroup has index $2$ and does not contain $e_3$.
		
\item[(ii)]  $\sigma(g_2\alpha \beta)=2e_3$. Then $\sigma(\alpha \beta)=e_1+e_3$, $\alpha=e_2+\alpha_3e_3,	\beta =e_1+e_2+(2n+1-\alpha_3)e_3$.

Interchanging $\alpha$ and $\beta$ if necessary, we may assume that $\alpha_3\le\beta_3$. Then {\it Argument ($*$)} implies
\[
G^{\bullet}\setminus\Sigma(S)
\subseteq e_1+\langle\alpha_1e_1+\alpha_2e_2,e_3\rangle.
\]
Since $(G:\langle\alpha_1e_1+\alpha_2e_2,e_3\rangle)=2$ and $e_1\notin\langle\alpha_1e_1+\alpha_2e_2,e_3\rangle$, we are done.
\end{itemize}
	
	\smallskip
	\noindent
	CASE 3.2:  $s=1$.

Then $S=R'g_2g_3\alpha$, where $g_2, g_3$ are distinct with $\{g_2, g_3\} \subseteq  \{e_1+e_3, e_2+e_3, e_1+e_2+e_3\}$. We discuss one case in detail, the others run along the same lines. Suppose that $g_2=e_1+e_3, g_3=e_2+e_3$. We distinguish three cases.
\begin{itemize}
\item[(i)] $\alpha_1=\alpha_2= 0$. Then $\alpha_3=2$ and $\alpha=2e_3$,  $\Sigma(Rg_1\alpha)\supseteq \langle e_3\rangle^{\bullet}$, whence  $G^{\bullet}= \Sigma (S)$.
		
\item[(ii)] $\alpha_1+\alpha_2= 1$,  say $\alpha_1=1 $ and $\alpha_2= 0$. Then $\sigma(g_2\alpha)=(1+\alpha_3)e_3\in \{e_3,2e_3\}$, a contradiction to $\alpha_3 \in [2, 2n-1]$.
		
\item[(iii)] $\alpha_1=\alpha_2= 1$.  Then $\sigma(g_2g_3\alpha)=(2+\alpha_3)e_3\in \{e_3,2e_3\}$ and $\alpha_3=2n-1$.

{\it Argument ($*$)} gives $G^{\bullet}\setminus\Sigma(S)=\{e_1-e_3,e_2-e_3,-e_3\}\subseteq-e_3+\langle e_1,e_2,2e_3\rangle$, where $(G:\langle e_1,e_2,2e_3\rangle)=2$. \qedhere
\end{itemize}
\end{proof}

\begin{theorem} \label{4.7}
Let $G = C_2 \oplus C_2 \oplus C_{2n}$ with $n\geq 1$. Then $\nu (G) = \nu_2 (G) = \mathsf d (G)-1$.
\end{theorem}

\begin{proof}
If $n$ is a power of $2$, then $G$ is a $2$-group and the claim follows from  Theorem \ref{3.5}. Suppose that  $n\ge 3$. 
	By \eqref{basic-inequ-2}, it suffices to prove that $\nu_2 (G) \le \mathsf d (G)-1$.
We show that $\mathsf d (G)-1$ satisfies the Property  \eqref{def-nu_p(G)}. 

Let $T \in \mathcal F(G)$ be a zero-sum free sequence with $|T|=\mathsf d (G)-1$. We have to prove that there is a subgroup $K \subsetneq G$ with $(G : K)=2$ and some $\alpha \in G \setminus K$ such that
\[
G^{\bullet} \setminus \Sigma (T) \subseteq \alpha + K \,.
\]

If $|G^{\bullet}\setminus \Sigma (T)| \le 1$, then the claim follows from Proposition \ref{2.6}, Corollary \ref{4.4}.3 and Lemma \ref{4.5}. 

\smallskip
Suppose that  $|G^{\bullet}\setminus \Sigma (T)|=2$, say  $G^{\bullet}\setminus \Sigma (T)=\{g,h\}$. 

If $g,h\in \langle h-g \rangle$,  say $g=k(h-g)$ and $h= (k+1) (h-g)$ for some $k \in \N$, then either  $g \in 2G$ or $h \in 2G$. Proposition \ref{2.6} implies that $-g \in \supp (U)$ (or $-h \in \supp (U)$) for some $U \in \mathcal A_{\max} (G)$. But $\supp (U) \cap 2G = \emptyset$ by Corollary \ref{4.4}.3 for all $U\in \mathcal A_{\max} (G)$,  a contradiction.

So we assume that $g,h\notin \langle h-g \rangle$. By Corollary \ref{4.4}.3, none of the elements of a minimal zero-sum sequence of length $\mathsf D(G)$ belongs to $2G$. Thus the claim follows from Lemma \ref{4.5}(c).

Suppose that $|G^{\bullet}\setminus \Sigma (T)| \ge 3$. We distinguish two cases and note that both cases can actually occur. Indeed, consider some $U \in \mathcal A_{\max} (G)$ as given in Proposition \ref{4.2}. After cancelling two elements of $U$ we get a zero-sum free sequence and the two elements can be chosen in such a way that both cases can occur.

\smallskip
\noindent
CASE 1. $\langle \supp (T) \rangle$ is a proper subgroup of $G$.

We set $G_1 = \langle \supp (T) \rangle$ and observe that 
\[
\mathsf d (G) -1 \ge \mathsf d (G_1) \ge |T| = \mathsf d (G) -1 \,,
\]
whence $\mathsf d (G_1) = |T|$ and $\Sigma (T) = G_1^{\bullet}$. Furthermore, we obtain that $\mathsf d (G) = \mathsf d (G_1)+1$. We claim $(G:G_1)=2$.  Otherwise, if $|G/G_1|\ge 3$, then there exist $a_1, a_2\in G\setminus G_1$ where $a_1- a_2\notin G_1$. If $2a_1\in G_1$, then $a_1+ a_2\in G\setminus G_1$ and $Ta_1 a_2$ is a zero-sum free sequence over $G$ of length $\mathsf d (G_1)+2$. If $2a_1\notin G_1$, then $Ta_1 ^2$ is a zero-sum free sequence over $G$ of length $\mathsf d (G_1)+2$, a contradiction. 

So $G = G_1 \sqcup (a+G_1)$ for some $a \in G \setminus G_1$. Thus, we get
\[
G^{\bullet} \setminus \Sigma (T) = G^{\bullet} \setminus G_1^{\bullet} \subseteq a + G_1 \,.
\]

\smallskip
\noindent
CASE 2. $\langle \supp (T) \rangle = G$.

Choose $U\in\mathcal A_{\max}(G)$ with $T\mid U$, and suppose that $U$ has one of the types (a)--(f) with respect to a basis $(e_1,e_2,e_2)$. If $U$ has type (a), (b*), (c), or (e), the assertion follows from Lemma \ref{4.5}(e). It remains to handle the cases when $U$ has type (f) or type (d*,d).

First, suppose that $U$ has type (f). Then $U$ has $2n$ terms in $e_3+\langle e_1,e_2\rangle$, and hence $T$ has a subsequence $R$ of length at least $2n-2$ with $(-e_3)+R\in\mathcal F(\langle e_1,e_2\rangle)$. Thus Lemma \ref{4.6} implies that assertion.

Second, suppose that $U$ has type (d*,d). We set
\[
U=e_3^{2n-1-2v}(e_2+e_3)^{2v}e_2(e_1+xe_3)\big(e_1+e_2+(1-x)e_3\big) \,,
\]
where $v\in[0,n-1]$, and $UT^{-1}=g_1g_2$. The sequence $U$ has $2n-1$ terms in $e_3+\langle e_1,e_2\rangle$. If $g_1$ and $g_2$ do not both belong to $\{e_3,e_2+e_3\}$, then $T$ has at least $2n-2$ terms in this coset, and Lemma \ref{4.6} implies the assertion.

Finally suppose that $g_1,g_2\in\{e_3,e_2+e_3\}$. Then $T$ contains the unique term $e_2\in B(U)$ and both terms of $C(U)$. Let $W\in\mathcal A_{\max}(G)$ with $T\mid W$. By Lemma \ref{4.3}, the partitions associated with $U$ and $W$ coincide. Since $B(T)\ne\emptyset$ and $|C(T)|=2$, Table \ref{BOX1} shows that $W$ has type (d*,d), and moreover $B(W)=B(T)=\{e_2\}$ and $C(W)=C(T)$. Suppose that  $W$ has type (d*,d) with respect to a basis  $(f_1,f_2,f_3)$. Then $f_2=e_2$, while the sum of the two terms in $C(T)$ gives
\[
e_2+e_3=\sum_{c\in C(T)}c=f_2+f_3.
\]
Consequently, $f_3=e_3$, and $WT^{-1}$ consists of two terms from $\{e_3,e_2+e_3\}$. Therefore, Proposition \ref{2.6}  yields
\[
G^{\bullet}\setminus\Sigma(T)\subseteq-\{e_3,e_2+e_3\}
\subseteq-e_3+\langle e_1,e_2,2e_3\rangle,
\]
where $(G:\langle e_1,e_2,2e_3\rangle)=2$. 
\end{proof}

\smallskip
\section{On groups of the form  $C_2^4 \oplus C_{2n}$} \label{5}
\smallskip

The goal in this section is to prove that $\nu (C_2^4 \oplus C_{2n}) = \nu_2 (C_2^4 \oplus C_{2n}) = \mathsf d (C_2^4 \oplus C_{2n}) - 1$ for $n > 70$ odd (Theorem \ref{5.5}). The groups $C_2^4 \oplus C_{2n}$ and $C_2^5 \oplus C_{2n}$, with $n$ odd and large, are the only groups  with $\mathsf d (G) > \mathsf d^* (G)$ for which the precise value of $\mathsf d (G)$ has been determined (see \cite{Sa-Ch14a, Zh23a}).

We start with the characterization of minimal zero-sum sequences of maximal length over $C_2^4 \oplus C_{2n}$.

\begin{proposition} \label{5.1}
Let $G=C_2^4\oplus C_{2n}$ with $n>70$ odd. A sequence $U \in \mathcal F (G)$   is a minimal zero-sum sequence of  length $\mathsf D(G)$ if and only if there exists a basis $(e_0, e_1,e_2,e_3,e_4)$ of $G$, with  $\ord (e_0)=2n$ and $\ord (e_i)=2$ for $i \in [1,4]$  such that $U=e_0^{2n-3}S_1S_2$, where
\[
\begin{aligned}
S_1 & = (e_1+e_2+e_3+\frac{n-1}{2} e_0)\prod_{i=1}^{3}(e_i+\frac{n+1}{2} e_0), \quad \text{and} \\
S_2 & = (e_1+e_2+e_3+e_4+\frac{n+1}{2}e_0)\prod_{i=1}^{3}(e_i+e_4+\frac{n+1}{2} e_0) \,.
\end{aligned}
\]
\end{proposition}

\begin{proof}
See \cite{Sa-Ch14a}.
\end{proof}

\begin{lemma} \label{5.2}
Let $G = C_2^4$ and let $A \subseteq G$ be a subset with $|A|=8$, say $A = \{a_1, \ldots, a_8\}$. Suppose that $A$ has the following property {\bf P}. 
\begin{itemize}
\item[{\bf P.}]If $a \in A$ is a sum of elements from $A$, say $a = \sum_{i \in I} a_i$ for some $I \subseteq [1,8]$, then $|I|$ is odd.
\end{itemize}    
If $B \subseteq A$ with $|B|=6$, then $A$ is the only subset of $G$, that has $8$ elements, contains $B$, and satisfies Property {\bf P}.
\end{lemma}

\begin{proof}
If $0\in A$ and $a \in A \setminus \{0\}$, then $a=0+a\in A$ is a sum of two elements of $A$, a contradiction.
Thus,  $0\notin A$ and $a_1+a_i\notin A$ for $i \in [2,8]$. Therefore, the set $\bar{A}=\{0,a_1+a_2,...,a_1+a_8 \}$ has $8$ elements and $A \cap \bar A=\emptyset$, whence  $A\sqcup \bar{A}=G$. 

Let $i, j \in [2,8]$ be distinct. Since $0+(a_1+a_i)\in \bar{A}$,  $(a_1+a_i)+(a_1+a_j)=a_i+a_j \in G \setminus A =  \bar{A}$, and $(a_1+a_i)+(a_1+a_i)=0\in \bar{A}$, it follows that $\overline{A}$ is a subgroup of $G$ and $A = a + \overline A$ for every $a \in A$.

Let $B \subseteq A$ with $|B|=6$, and let $b \in B$. Then $-b+B$ is a subset of a subgroup $G_1$ of $G$ with $|G_1|=8$, $|-b+B|=6$, and $0 \in -b+B$.

	On the other hand, a subgroup of order at most $4$
	contains at most $4$ elements, while $|-b+B|=6$.
	Hence $|\langle -b+B\rangle|\ge 8$ and $-b+B$ uniquely determines the subgroup $G_1$.
	
	Since $|\overline A|=8$ and $\langle -b+B\rangle \subseteq \overline A$, we obtain
	$\langle -b+B\rangle=\overline A$.
	Thus $B$ uniquely determines $\overline A$, and $A=G\setminus \overline A$ is ensured.
\end{proof}

\begin{proposition} \label{5.3}
Let $G = C_2^4 \oplus C_{2n}$ with $n>70$  odd. Let $U = e_0^{2n-3}S_1S_2 \in \mathcal A_{\max} (G)$ with all notation as in Proposition \ref{5.1}, and let  $\varphi \colon G\to G/\langle e_0 \rangle \cong C_2^4$ be the canonical epimorphism.
\begin{enumerate}
\item The element $e_0 \in G$ is the only element $g \in G$ with multiplicity $\mathsf v_g (U) > 1$.
		
\item $\supp \big( \varphi(S_1S_2) \big)$ satisfies Property {\bf P} of Lemma \ref{5.2}.
\end{enumerate}
\end{proposition}

\begin{proof}
1. This is obvious by Proposition \ref{5.1}.
	
2. Let $(e_0, e_1,e_2,e_3,e_4)$ be a basis of $G$ such that $U = e_0^{2n-3}S_1S_2$ with all notation as in Proposition \ref{5.1}. Then 
 $(e_1+\langle e_0 \rangle, e_2+\langle e_0 \rangle, e_3+\langle e_0 \rangle, e_4+\langle e_0 \rangle )$ is a basis of $G/\langle e_0 \rangle$, and $A:=\supp \big( \varphi(S_1S_2) \big) \subseteq G/\langle e_0 \rangle$ has eight elements. We set
\[
\begin{aligned}
E_0 & =\{0+\langle e_0 \rangle, e_1+e_3+\langle e_0 \rangle,e_2+e_3+\langle e_0 \rangle,e_1+e_2+\langle e_0 \rangle\} \quad \text{ and} \\ 
A_0 & = (e_1+\langle e_0 \rangle) + E_0 =\{e_1+\langle e_0 \rangle,e_2+\langle e_0 \rangle,e_3+\langle e_0 \rangle,e_1+e_2+e_3+\langle e_0 \rangle\} ,
\end{aligned}
\]
and we observe  that the sumset $2A_0=2E_0=E_0$ and $A_0+E_0=A_0$. Next we set
\[
E_4=\{0+\langle e_0 \rangle,e_4+\langle e_0 \rangle\}, \ E=E_4+E_0.
\]
We observe that $A=E_4+A_0$ and $E=a+A$ for every $a\in A$.
	Thus $2A(=2A_0+E_4)=2E(=2E_0+E_4)=E$ and $A+E(=A_0+E_0+E_4=A_0+E_4)=A$. It is easy to see that $A\sqcup E=G/\langle e_0\rangle$ and that $A$ is a coset of $E$.
	
Thus, for every $k \in \N$, we have
 $(2k)A = kE=E$ and  $(2k+1)A=kE+A=A$, whence $A$ satisfies Property {\bf P}. 
\end{proof}

\begin{lemma}\label{5.4}
	Let $G = C_2^4 \oplus C_{2n}$ with $n\geq 3$ odd, and let $T \in \mathcal F (G\setminus 2G)$ be zero-sum free. In each of the following two cases,  $G^{\bullet} \setminus \Sigma (T) \subseteq \alpha + H$ for some subgroup $H \subseteq G$ with $(G \colon H)=2$ and some $\alpha \in G \setminus H$. 
\begin{enumerate}
\item[(a)]  $|G^{\bullet} \setminus \Sigma (T)| \le 1$, with $G^{\bullet} \setminus \Sigma (T)\subseteq G\setminus 2G$.
		
\item[(b)]  $G^{\bullet}\setminus \Sigma (T)=\{g,h\}\subseteq G\setminus 2G$. 
\end{enumerate}  
\end{lemma}

\begin{proof}
(a)  If $G^{\bullet} \setminus \Sigma (T) = \emptyset$, the claim is obvious. Suppose that $G^{\bullet} \setminus \Sigma (T) = \{\alpha\}$. Assume to the contrary that $\alpha$ is contained in all subgroups $H \subseteq G$ with $(G \colon H)=2$. Since the intersection of all these subgroups equals $2G$, it follows that $\alpha \in 2G$, a contradiction.

(b)
	We have proved that $g,h\notin 2G$ implies $g,h\notin \langle h-g \rangle$ in Lemma \ref{4.5}(c).
	Let $
	\pi \colon G \longrightarrow G/2G \cong C_2^5$ 
	be the canonical homomorphism. 
	Set $
	\overline g=\pi(g)$ and $\overline h=\pi(h)$.
	Since $g,h\notin 2G$, both $\overline g$ and $\overline h$
	are nonzero.
	
	We claim that there exists a linear functional
	\[
	\lambda\colon G/2G\longrightarrow \mathbb F_2
	\]
	such that $
	\lambda(\overline g)=\lambda(\overline h)=1.$
	Indeed, if $\overline g=\overline h$, then we extend
	$\overline g$ to a basis of $G/2G$, define $\lambda(\overline g)=1$,
	and define $\lambda$ to be zero on all the other basis vectors.
	
	Suppose that $\overline g\neq\overline h$. Then
	$\overline g$ and $\overline h$ are linearly independent over
	$\mathbb F_2$. In fact, if they were linearly dependent, then,
	since both are nonzero and the only nonzero scalar in $\mathbb F_2$
	is $1$, we would have $\overline g=\overline h$, a contradiction.
	We may therefore extend
	$\{\overline g,\overline h\}$ to a basis of $G/2G$, 
	and define $\lambda$ to be zero on all the remaining basis vectors.
	This proves the claim.
	
	Now let
	\[
	\chi=\lambda\circ\pi\colon G\longrightarrow\mathbb F_2
	\quad\text{and}\quad
	H=\ker\chi.
	\]
	Since $\chi(g)=1$, the homomorphism $\chi$ is nonzero and hence
	surjective. Consequently, $
	(G:H)=|\mathbb F_2|=2.$
	Moreover, $
	\chi(g)=\chi(h)=1,$
	and therefore $
	g,h\in g+H.$
	We obtain $
	G^{\bullet}\setminus\Sigma(T)
	=\{g,h\}\subseteq g+H$,
	where $(G:H)=2$ and $g\in G\setminus H$.
\end{proof}

\begin{theorem}\label{5.5}
Let $G=C_2^4\oplus C_{2n}$ with $n>70$  odd. Then $\nu(G)=\nu_2 (G)=\mathsf d (G)-1$.
\end{theorem}

\begin{proof}
By \eqref{basic-inequ-2}, it suffices to prove that $\nu_2 (G) = \mathsf d (G)-1$.
By Proposition \ref{5.1}, we have  $\mathsf d (G)-1 = 2n+3$.
	We show that $\mathsf d (G)-1$ satisfies the Property  \eqref{def-nu_p(G)}. Since $\nu_2 (G)$ is the smallest number satisfying this property and since $\mathsf d (G)-1 \le \nu_2 (G) \le \mathsf d (G)$ by the Inequality \eqref{basic-inequ-2}, it follows that $\nu_2 (G) = \mathsf d (G)-1$.
	
	Let $T \in \mathcal F (G)$ be  zero-sum free with $|T|=\mathsf d (G)-1$. We have to prove that there is a subgroup $H \subsetneq G$ with $(G \colon H) = 2$ and some $\alpha \in G \setminus H$ such that
	\[
	G^{\bullet} \setminus \Sigma (T) \subseteq \alpha + H \,.
	\]
We use Proposition \ref{2.6}. If $|\{U \in \mathcal A_{\max}(G) \colon T \mid U \}|=0$, then $G^{\bullet} \setminus \Sigma (T)=\emptyset$ and we are done.

So we may assume that there exists $U \in \mathcal A_{\max} (G)$ with $T\mid U$ (we use all notation of Proposition \ref{5.1}). Then, by Proposition \ref{5.1}, there is only one element $g \in G$ with $\mathsf v_g (U) = \mathsf h (U) = 2n-3$. Moreover, Proposition \ref{5.1} gives $T\in\mathcal F(G\setminus2G)$.
We distinguish two cases.

\smallskip
\noindent
CASE 1.  $\mathsf h (T)\le 2n-4$.

Let $U \in \mathcal A_{\max} (G)$ with $T \mid U$. 
	By Proposition \ref{5.1},
	there exists a unique element
	$g\in G$ such that
	$\mathsf v_g(U)=2n-3>1$.
	Since $\mathsf h(T)\le 2n-4$,
$g \mid UT^{-1}$, whence $U=Tg(-\sigma(Tg))$. 

We use Proposition \ref{2.6} twice. If $\mathsf v_g(T)=2n-5$, then $g= -\sigma(Tg)$, and $U$ is the unique sequence $T\mid U$. So
	$$
	G^{\bullet}\setminus \Sigma (T)  = \supp \big(-(UT^{-1}) \big)
	=\{ -g\} \subseteq G \setminus 2G,
	$$
whence Lemma \ref{5.4}(a) implies the assertion. 
	
	Otherwise, if $\mathsf v_g(T)=2n-4$, then $g\ne -\sigma(Tg)$. By Proposition \ref{5.1}, we infer that  $-\sigma(Tg)\mid S_1S_2$, and $U$ is the unique sequence $T\mid U$. So
\[
	G^{\bullet}\setminus \Sigma (T)  = \supp \big(-(UT^{-1}) \big)
	=\{ -g,\sigma(Tg)\}
	\subseteq G \setminus 2G \,.
\]
Thus
Lemma \ref{5.4}(b) implies the assertion.
	
\smallskip
\noindent
CASE 2.	 $\mathsf h (T)= 2n-3$.

Let $g \in G$ with $\mathsf{v}_{g}(T)=2n-3$ and note that $\ord(g)=2n$. Let $U \in \mathcal A_{\max} (G)$ with $T \mid U$. By Proposition \ref{5.1}, we have $U = g^{2n-3}S_1S_2$ with all notation as in Proposition \ref{5.1}.

Let $\varphi \colon G\to G/\langle g \rangle$ be the canonical epimorphism. Let
	 $A=\supp \big( \varphi (S_1S_2) \big)\subseteq G/\langle g \rangle $
	
		, and let
		$B=\supp\big(\varphi(Tg^{3-2n})\big).$
		Since $|B|=6$ and
		$B\subseteq A=\supp(\varphi(S_1S_2))$,
		Lemma \ref{5.2} implies that
		$A$ is uniquely determined by $B$.
		In particular, $A$ does not depend on the choice of $U$.
	
	We set $A=\supp \big(\varphi(Tg^{3-2n}) \big)\cup \{ u_1+\langle g \rangle,u_2+\langle g \rangle\}$ where $u_j\in G$ with $\ord(u_j)=2$ for $j \in [1,2]$. Let $(u_1,u_2,e_3,e_4, g)$ be a basis of $G$ with $\ord (e_3)=\ord (e_4)=2$. 
	
		(By Proposition \ref{5.1},
		among the two elements of
		$A\setminus \supp(\varphi(Tg^{3-2n}))$,
		at most one may arise from the unique term whose
		$g$-coefficient equals $(n-1)/2$,
		while all remaining terms have
		$g$-coefficient $(n+1)/2$.
		Hence the sum of the two missing terms is either $
		u_1+u_2+ng$ or $
		u_1+u_2+(n+1)g.
		$
		Therefore these are the only two possibilities for
		$-\sigma(T)$.)

	If $-\sigma(T)= u_1+u_2+ ng$, then 
\[
U \in \Big\{T(u_1+\frac{n+1}{2}g)(u_2+\frac{n-1}{2}g) , \ T(u_2+\frac{n+1}{2}g)(u_1+\frac{n-1}{2}g) \Big\}
\]
by Proposition \ref{5.1} and hence, by Proposition \ref{2.6},  
	$$
	G^{\bullet}\setminus \Sigma (T) 
	\subseteq\{u_1-\frac{n+1}{2}g,u_2-\frac{n+1}{2}g,u_1-\frac{n-1}{2}g,u_2-\frac{n-1}{2}g,\}
	\subseteq u_1+\langle u_1+u_2, g \rangle.
	$$
Since $(u_1, u_1+u_2,e_3,e_4, g)$ is a basis of $G$, we have $u_1 \notin \langle u_1+u_2, g \rangle$. Let $ K :=\langle u_1+u_2,e_3,e_4, g \rangle$. Then $(G:K)=2$ and $u_1\notin K$. It follows that $
G^\bullet\setminus\Sigma(T)
\subseteq u_1+K $.

If $-\sigma(T)= u_1+u_2+ (n+1)g$, then $U=T(u_1+\frac{n+1}{2}g)(u_2+\frac{n+1}{2}g)$ by Proposition \ref{5.1} and hence, by Proposition \ref{2.6},  
\[
	G^{\bullet}\setminus \Sigma (T)  
	=\{u_1-\frac{n+1}{2}g,u_2-\frac{n+1}{2}g\}
	\subseteq (u_1-\frac{n+1}{2}g)+\langle u_1+u_2\rangle \,. 
\]
Since $(u_1,u_2,e_3,e_4,g)$ is a basis, both displayed elements lie in $G\setminus2G$; hence the claim follows from Lemma \ref{5.4}(b).
\end{proof}

\smallskip
\section{A local variant of $\nu (G)$} \label{6}
\smallskip
	
In this section  we introduce a local variant of $\nu (G)$. It is easy to verify that the maximum of the local variants equals $\nu (G)$ (see Lemma \ref{6.2}). They allow a finer analysis of the behavior of $\nu (G)$, even in cases when the precise value of $\mathsf d (G)$ is unknown. We study these new  invariants for a group $G$ with $\mathsf d (G) > \mathsf d^* (G)$ (Theorem \ref{6.5}) and for $G = C_n^3$ with respect to a special type of zero-sum free sequences (Theorem \ref{6.6} and the preceding discussion). 
As mentioned in the Introduction, Gao (\cite[Section 4]{Ga00b}) conjectured that $\nu (G) = \mathsf d (G)-1$ for all nontrivial finite abelian groups $G$. In light of our current understanding of the Davenport constant (and of related zero-sum constants, such as the {E}rd{\H{o}}s-{G}inzburg-{Z}iv constant), this conjecture seems to be out of reach. A more realistic goal is to study $\nu (G)$ for groups with $\mathsf d (G) = \mathsf d^* (G)$ ($p$-groups and rank two groups have this property; for more see Proposition \ref{2.2},  \cite{Bh-SP07a}, \cite[Theorem 3.3.10]{Ge-Gr-Zh26a})). Theorems \ref{6.4} and \ref{6.6} are first steps in this direction.

\smallskip
Let $G$ be a finite abelian group and let   $S \in \mathcal F (G)$ be zero-sum free. Then $\Sigma (S) =G^{\bullet}$ if and only if 
\begin{itemize}
\item[] There is no zero-sum free sequence $T \in \mathcal F (G)$ with $S \mid T$ and $S \ne T$. 
\end{itemize}
If this  property holds, then $S$ is called a {\it maximal zero-sum free sequence} (with respect to $G$).

\begin{definition} \label{6.1}
Let $G$ be a nontrivial, finite abelian group and let   $S \in \mathcal F (G)$ be a maximal zero-sum free sequence. Then $\nu (S)$ is the smallest integer such that, for any subsequence $T$ of $S$,
\begin{equation} \label{def-nu(S)}
|T|\geq \nu(S)  \ \mbox{ implies } \ G^{\bullet} \setminus \Sigma (T) \subseteq \alpha+H \ \mbox {for some  subgroup $H  \subsetneq  G$ and some $\alpha\in G\setminus H$}.
\end{equation}
\end{definition}

\begin{lemma} \label{6.2}
Let $G$ be a nontrivial, finite abelian group.
\begin{enumerate}
\item For every maximal zero-sum free sequence $S \in \mathcal F (G)$, we have $|S|-1 \le \nu (S) \le |S|$.

\item $\nu (G) = \max \{ \nu (S) \colon S \in \mathcal F (G) \ \text{is maximal zero-sum free} \}$.
\end{enumerate}
\end{lemma}

\begin{proof}
1. Let $S \in \mathcal F (G)$ be maximal zero-sum free. Every subsequence $T$ of $S$ with $|T| \ge |S|$ is equal to $S$, whence $G^{\bullet} \setminus \Sigma (T)  = G^{\bullet} \setminus \Sigma (S) = \emptyset$. Thus $|S|$ satisfies Property  \eqref{def-nu(S)}, whence $\nu (S) \le |S|$. 

Assume to the contrary that $\nu (S) \le |S|-2$. Let $T \in \mathcal F (G)$ with $T \mid S$, $\nu (S) \le|T| = |S|-2$, and $G^{\bullet} \setminus \Sigma (T) \subseteq \alpha + H$ for some subgroup $H \subsetneq G$ and some $\alpha \in G \setminus H$. Then there are $g_1, g_2 \in G$ such that $g_1g_2T$ is a subsequence of $S$, whence zero-sum free. Thus, $-g_i \in G^{\bullet} \setminus \Sigma (T) \subseteq \alpha + H$ for $i \in [1,2]$, whence $-g_1-g_2 \in 2 \alpha + H$. Since $2 \alpha + H \ne \alpha + H$, it follows that $-g_1-g_2 \notin \alpha + H$. Since $g_1+g_2\ne0$ by the zero-sum freeness of $g_1g_2T$, this implies that $-g_1-g_2 \notin G^{\bullet}\setminus\Sigma(T)$, whence $-g_1-g_2\in\Sigma(T)$, and hence $g_1g_2T$ is not zero-sum free, a contradiction.

2. If $S \in \mathcal F (G)$ is maximal zero-sum free, then $\nu (G)$ satisfies Property \eqref{def-nu(S)}, whence $\nu (S) \le \nu (G)$. This implies that $\max \{ \nu (S) \colon S \in \mathcal F (G) \ \text{is maximal zero-sum free} \} \le \nu (G)$.

In order to show that $\nu (G) \le \max \{ \nu (S) \colon S \in \mathcal F (G) \ \text{is maximal zero-sum free} \}$, we need to verify that $\max \{ \nu (S) \colon S \in \mathcal F (G) \ \text{is maximal zero-sum free} \}$ satisfies Property \eqref{def-nu(G)}. Let $T \in \mathcal F (G)$ be zero-sum free with $|T| \ge \max \{ \nu (S) \colon S \in \mathcal F (G) \ \text{is maximal zero-sum free} \}$. There is a maximal zero-sum free sequence $T^*$ with $T \mid T^*$ and since
\[
|T| \ge \max \{ \nu (S) \colon S \in \mathcal F (G) \ \text{is maximal zero-sum free} \} \ge \nu (T^*) \,,
\]
the definition of $\nu (T^*)$ implies that $G^{\bullet} \setminus \Sigma (T)$ is contained in a proper coset of some subgroup $H$ of $G$. Thus $\max \{ \nu (S) \colon S \in \mathcal F (G) \ \text{is maximal zero-sum free} \}$ satisfies Property \eqref{def-nu(G)}.
\end{proof}

Lemma \ref{6.2} shows that $\nu (G) = \mathsf d (G)$ if and only if there is a zero-sum free sequence $S \in \mathcal F (G)$ with $|S|=\mathsf d (G)$ and $\nu (S) = |S|$. There are maximal zero-sum free sequences $S$ over prime cyclic groups $G$ with $\nu (S) = |S|$, but they all have length $|S| < \mathsf d (G)$.

\begin{lemma} \label{6.3}
Let $G$ be a nontrivial, finite abelian group, let $(e_1, \ldots, e_r)$ be a basis of $G$ with $r \ge 1$, and let $S = \prod_{i=1}^r e_i^{\ord (e_i)-1}$.
\begin{enumerate}
\item $S$ is maximal zero-sum free and $\nu (S) = |S|-1$.

\item If $G \cong C_{n_1} \oplus \ldots \oplus C_{n_r}$ with $ 1 < n_1 \mid \ldots \mid n_r$ and with $\ord (e_i)=n_i$ for $i \in [1,r]$, then $\nu (S) = \mathsf d^* (G) - 1$.
\end{enumerate}
\end{lemma}

\begin{proof}
1. We have
\[
\Sigma (S) = \Big\{ \sum_{i=1}^r k_ie_i \colon k_i \in [0, \ord (e_i)-1] \ \text{for $i \in [1,r]$ and } \ k_1+\ldots +k_r > 0 \Big\} \,.
\]
Since $(e_1, \ldots, e_r)$ is a basis of $G$, this implies that $\Sigma (S) = G^{\bullet}$. Next we assert that $|S|-1$ fulfills Property \eqref{def-nu(S)}. Since $|S|-1 \le \nu (S) \le |S|$ by Lemma \ref{6.2}, this implies that $\nu (S) = |S|-1$. To show Property \eqref{def-nu(S)}, we pick a subsequence $T$ of $S$ with $|T|=|S|-1$. Then there is some $i \in [1,r]$ with $T = e_i^{-1}S$, say $i=1$. Then
\[
G^{\bullet} \setminus \Sigma (T) \subseteq  (\ord (e_1)-1)e_1 + \langle e_2, \ldots, e_r \rangle \,.
\]

2. This is an immediate consequence of Part 1.
\end{proof}
 	
\smallskip
Let $G = C_{n_1} \oplus \ldots \oplus C_{n_r}$ with $1 < n_1 \mid \ldots \mid n_r$ and let $G_1 \subseteq G$ be a subgroup. Then
\begin{equation} \label{imp-1}
\mathsf d (G) \ge \mathsf d (G_1) + \mathsf d (G/G_1) \quad \text{and} \quad \mathsf d^* (G) \ge \mathsf d^* (G_1) + \mathsf d^* (G/G_1) \,.
\end{equation}
Suppose that $G_1 = \oplus_{i \in I} C_{n_i}$ with $I \subseteq [1,r]$. Then 
\begin{equation} \label{imp-2}
\mathsf d (G) = \mathsf d^* (G) \quad \text{implies that } \quad \mathsf d (G_1) = \mathsf d^* (G_1)
\end{equation}
(for \eqref{imp-1} and \eqref{imp-2}, see \cite[Lemma 3.2.1]{Ge-HK06a} and \cite[Lemma 4.1]{Gr-Ma-Or09}), and our next result shows, in particular,  that
\[
\mathsf d (G) = \mathsf d^* (G) \ \text{and} \ \nu (G) = \mathsf d (G)-1 \quad \text{imply that } \quad \nu (G_1) = \mathsf d (G_1) - 1 \,.
\]

\smallskip
\begin{theorem} \label{6.4}
Let $G, G_1, G_2$ be nontrivial, finite abelian groups with $G = G_1 \oplus G_2$ and with $\mathsf d (G) = \mathsf d (G_1) + \mathsf d^* (G_2)$. If $\nu (G) = \mathsf d (G)-1$, then $\nu (G_1) = \mathsf d (G_1) - 1$.
\end{theorem}

\begin{proof}
First, we outline that we can proceed by induction on the rank of $G_2$. Suppose that $G_2 = C_{n_1} \oplus \ldots \oplus C_{n_r}$ with $1 < n_1 \mid \ldots \mid n_r$. Then $\mathsf d^* (G_2) = \sum_{i=1}^r (n_i-1)$. 
If $G_2 = G_3 \oplus G_4$, where $G_3 = \sum_{i \in I} C_{n_i}$ and $G_4 = \sum_{i \in J} C_{n_i}$ with $I, J$ nontrivial and $I \sqcup J = [1,r]$, then 
\[
\mathsf d (G_3 \oplus G_4) \ge \mathsf d (G_3) + \mathsf d (G_4) \ge \mathsf d (G_3)+\mathsf d^* (G_4) \ge \mathsf d^* (G_3)+\mathsf d^* (G_4) = \mathsf d^* (G_3 \oplus G_4) \,.
\]
Furthermore, we get
\[
		\mathsf d (G)= \mathsf d (G_1) + \mathsf d^* (G_2) =\mathsf d (G_1) + \mathsf d^* (G_3) + \mathsf d^* (G_4)\le \mathsf d (G_1 \oplus G_3) + \mathsf d^* (G_4)\le \mathsf d (G).
\]
Then,
		$
		\mathsf d (G) = \mathsf d (G_1 \oplus G_3) + \mathsf d^* (G_4)$ and $\mathsf d (G_1 \oplus G_3) = \mathsf d (G_1) + \mathsf d^* (G_3) .$
		Since $\mathsf r (G_4) < \mathsf r (G_2)$ and $\mathsf r (G_3) < \mathsf r (G_2)$,  the assumption $\nu (G) = \mathsf d (G) - 1$ implies that $\nu (G_1 \oplus G_3) = \mathsf d (G_1 \oplus G_3) - 1$ and then $\nu (G_1) = \mathsf d (G_1) - 1$.  Thus it suffices to consider the case $r = 1$.
 Then the general case follows by induction. 

Suppose that $G_2 = \langle e \rangle$ with $\ord (e)=n \ge 2$, that $\nu (G) = \mathsf d (G)-1$, and assume to the contrary that  $\nu (G_1) = \mathsf d (G_1)$. By Lemma \ref{6.2}, there exists a maximal zero-sum free sequence $S_1 \in \mathcal F (G_1)$ with $\nu (S_1) = \nu (G_1) = \mathsf d (G_1)$ and a subsequence  $T_1\t S_1$ with $|T_1|=\mathsf d(G_1)-1$ for which there is no subgroup  $H_1 \subseteq G_1$ and no $\alpha\in G_1\setminus H_1$ such  that $G_1^{\bullet} \setminus \Sigma (T_1)\subseteq  (\alpha+H_1)$.

The sequence  $S=S_1e^{n-1} \in \mathcal F (G)$ is zero-sum free with $|S|=|S_1|+n-1=\mathsf d(G_1)+n-1=\mathsf d(G)$. Then $T = T_1e^{n-1} \t S$ and $|T| = \mathsf d (G)-1$.  Since $\nu (G) = \mathsf d (G)-1$, there exists some proper subgroup $H \subseteq G$ and $\alpha\in G\setminus H$ with 
\begin{equation} \label{assume} 
G^{\bullet} \setminus \Sigma (T)  \subseteq (\alpha+H) \,.
\end{equation}
Since $G=G_1\oplus \langle e\rangle$ with $\ord(e)=n$, we have  
\begin{equation} \label{Sigma=}
\Sigma (T)\cap G_1=\Sigma (T_1e^{n-1})\cap G_1=\Sigma (T_1) \,.
\end{equation}
If $(\alpha+H)\cap G_1=\emptyset$, then \eqref{assume} and \eqref{Sigma=} ensure that $G_1^{\bullet} \subseteq \Sigma (T)\cap G_1=\Sigma (T_1)$, in which case   $G_1^{\bullet} \setminus \Sigma (T_1) \subseteq (\beta+H_1)$ with $H_1 \subseteq G_1$ the trivial group and $\beta\in G_1\setminus H_1$ any nonzero element, a contradiction. 

Therefore we can instead assume 
\[
(\alpha+H)\cap G_1\neq \emptyset \,.
\]
As a result, we may replace $\alpha$ by an alternative representative for the coset $\alpha+H$ that lies in $G_1$, whence we may  assume that $\alpha\in G_1$. Consequently, $(\alpha+H)\cap G_1=\alpha+(H\cap G_1)$.

Recall that the intersection of an $H$-coset and a $K$-coset is always either empty or equal to an $H\cap K$-coset. 
Indeed, any $g\in (\alpha+H)\cap (\beta+K)$ is a representative both for $\alpha+H$ and $\beta+K$, allowing us to w.l.o.g. assume $g=\alpha=\beta$, whence $(\alpha+H)\cap (\beta+K)=(g+H)\cap (g+K)=g+(H\cap K)$.

If $m= (G_1+H:H)$, then
 $G_1+H$ can be written as a disjoint union of $m$ $H$-cosets, say $G_1+H=\bigsqcup_{i=0}^{m-1}(\alpha_i+H)$. Each coset $\alpha_i+H$ intersects $G_1$, so as above, we can  assume $\alpha_0,\ldots,\alpha_{m-1}\in G_1$ with $\alpha_0=0$. Then  
\begin{equation} \label{disjoint}
G_1=\bigsqcup_{i=0}^{m-1}\Big(\alpha_i+(H\cap G_1)\Big) \,.
\end{equation}
Since $\alpha\in G_1$ but $\alpha\in G\setminus H$, it follows that $H$ does not contain the subgroup $G_1$, whence  $m\geq 2$. Relabeling the cosets, we may assume $\alpha=\alpha_1$. But now \eqref{assume} and \eqref{Sigma=} imply that 
\begin{align*}
			\alpha_i + (H \cap G_1) = (\alpha_i + H) \cap G_1 \subseteq \Sigma (T) \cap G_1^{\bullet} = \Sigma (T_1)
			\qquad &\text{for all } i \in  [2, m-1],\\
			(\alpha_i + (H \cap G_1))^{\bullet} = (H \cap G_1)^{\bullet} \subseteq \Sigma (T) \cap G_1^{\bullet} = \Sigma (T_1)
			\qquad &\text{for } i =0,
\end{align*}
which in view of \eqref{disjoint} means $G_1^{\bullet} \setminus \Sigma (T_1) \subseteq \alpha_1+(H\cap G_1)$. Since $\alpha_1+H\neq \alpha_0+H=H$, we have $\alpha_1\notin H\cap G_1$, whence  $H_1 =H\cap G_1 \subseteq G_1$ contradicts the hypotheses assumed to hold for $T_1$. 
\end{proof}

By \cite[Theorem 4.1]{Sc11b}, we have $\mathsf d (C_6^2 \oplus C_{6n}) = \mathsf d^* (C_6^2 \oplus C_{6n})$ for all $n \in \N$.
It was previously unknown that the group  $G = C_2^2 \oplus C_6^2 \oplus C_{6n}$ with $n \in \N$ has the property that $\mathsf d (G) > \mathsf d^* (G)$.

\begin{theorem} \label{6.5}
Let $G = C_2^2 \oplus C_6^2 \oplus C_{6n}$ with $n \in \N$, and let $(e_1, \ldots, e_5)$ be a basis of $G$ with $\ord (e_1) = \ord
(e_2) = 2$, $\ord (e_3) = \ord (e_4) = 6$ and $\ord (e_5) = 6n$.
\begin{enumerate}
\item The sequence 
      \begin{align*}
      S =   g_1 g_2 g_3 g_4 g_5 g_6  g_7^3 g_8^4 g_9^3 g_{10}^{6n-4} 
      \end{align*}
      is zero-sum free, where $g_1 = e_3+e_4+e_5$, $g_2 = e_2+e_5$, $g_3 =
      e_1+e_4$, $g_4 = e_2+e_4$, $g_5 = e_1+e_3$, $g_6=e_2+e_3$, $g_7 =
      e_1+e_2+2e_4+e_5$, $g_8 = 2e_3+e_4$, $g_9 = e_2+2e_3$ and $g_{10} =
      e_1+e_5$.  In particular, $\mathsf d (G) \ge |S|  = 6n+12 = \mathsf d^* (G)+1$.
      
\item $S$ is maximal zero-sum free.

\item $\nu (S) = |S| -1$.
\end{enumerate}      
\end{theorem}

\begin{proof}
1. Suppose that  $T = \prod_{i=1}^{10} g_i^{l_i}$ is a zero-sum
subsequence of $S$ with $l_i \in [0,1]$ for $i \in [1,6]$, $l_7, l_9
\in [0,3]$, $l_8 \in [0,4]$ and $l_{10} \in [0, 6n-4]$. We have to
show that $T = 1$. We start with the following system of
congruences:
\begin{align*}
\begin{aligned}
a_1 = l_3+l_5+l_7+l_{10} & \equiv 0 \mod 2 \\
a_2 = l_2+l_4+l_6+l_7+l_9 & \equiv 0 \mod 2 \\
a_3 = l_1+l_5+l_6+ 2l_8 + 2l_9 & \equiv 0 \mod 6 \\
a_4 = l_1+l_3+l_4+2l_7+l_8 & \equiv 0 \mod 6  \quad \text{and} \\
a_5 = l_1+l_2+l_7+l_{10} & \equiv 0 \mod 6n
\end{aligned}
\end{align*}
Note that $a_3 \in [0, 17]$, $a_4 \in [0, 13]$ and $a_5 \in [0,
6n+1]$. We distinguish the following two cases.

\smallskip
\noindent CASE 1: \,$a_5 = 0$.

Then $l_1=l_2=l_7=l_{10} = 0$ and $a_4 \in \{0, 6\}$.

Suppose that $a_4 = l_3+l_4+l_8=0$. Then $l_3=l_4=l_8=0$,
$a_1=l_5=0$, $a_2 = l_6+l_9 \equiv 0 \mod 2$ and $a_3 = l_6 + 2l_9
\equiv 0 \mod 6$. Thus $l_6 = 0$, $l_9 \equiv 0 \mod 2$ and $2 l_9
\equiv 0 \mod 6$. Therefore $l_9 = 0 = \sum_{i=1}^{10} l_i = |T|$.

Suppose that $a_4 = l_3+l_4+l_8=6$. Then $l_3=l_4=1$, $l_8 = 4$,
$l_5 = a_1-l_3=1$, $a_2 = 1 + l_6 + l_9 \equiv 0 \mod 2$ and $a_3 =
1 + l_6 + 8 + 2l_9 \equiv 0 \mod 6$. Then the last equation implies
 $l_6 = 1$, whence $a_2 \equiv l_9 \equiv 0 \mod 2$,  $4 + 2 l_9
\equiv 0 \mod 6$, and thus $l_9=4>3$, a contradiction.

\smallskip
\noindent CASE 2: \,$a_5 = 6n$.

Then $l_{10} \in \{6n-5, 6n-4\}$, and we distinguish two cases.

\smallskip
\noindent CASE 2.1: \,$l_{10} = 6n-5$.

Then $l_1=l_2=1$ and $l_7 = 3$. Then $a_4=1+l_3+l_4+6+l_8 \equiv 0
\mod 6$ and hence $l_3+l_4+l_8=5$.

If $l_8=3$, then $l_3=l_4=1$, $a_1=1+l_5+3+6n-5 \equiv 0 \mod 2$,
$l_5=1$, $a_2=1+1+l_6+3+l_9 \equiv 0 \mod 2$ and $a_3=1+1+l_6+6+2l_9
\equiv 0 \mod 6$. The last equation implies $l_6=0$, hence $1+l_9
\equiv 0 \mod 2$ and $2+2l_9 \equiv 0 \mod 6$, a contradiction.

Suppose that $l_8=4$. If $l_4=0$, then $l_3=1$, $a_1=l_3+l_5 \equiv
0 \mod 2$, $l_5=1$, $a_2=1+0+l_6+3+l_9 \equiv 0 \mod 2$,
$a_3=1+1+l_6+8+2l_9 \equiv 0 \mod 6$, hence $2 \t l_9$, $l_6=0$ and
$10+2l_9 \equiv 0 \mod 6$, a contradiction. If $l_4=1$, then
$l_3=0$, $a_1=l_3+l_5 \equiv 0 \mod 2$, $l_5=0$, $a_2=1+1+l_6+3+l_9
\equiv 0 \mod 2$, $a_3=1+0+l_6+8+2l_9 \equiv 0 \mod 6$, hence
$l_6=1$, $2 \t l_9$ and $10+2l_9 \equiv 0 \mod 6$, a contradiction.

\smallskip
\noindent CASE 2.2: \,$l_{10} = 6n-4$.

Then $l_7 \in \{2,3\}$.

\smallskip
\noindent CASE 2.2.1: \,$l_{7} = 2$.

Then $l_1=l_2=1$, $2 \t a_1$, $l_3=l_5$, $a_4 = 1+l_3+l_4+4+l_8
\equiv 0 \mod 6$ and hence $l_3+l_4+l_8=1$.

If $l_3=1$, then $l_4=l_8=0$, $a_2=1+l_6+2+l_9 \equiv 0 \mod 2$,
$a_3=1+1+l_6+2l_9 \equiv 0 \mod 6$, $l_6=0$, $l_9$ odd and $2+2l_9
\equiv 0 \mod 6$, a contradiction.

If $l_4=1$, then $l_3=l_8=0$, $a_2=1+1+l_6+2+l_9 \equiv 0 \mod 2$,
$a_3 = 1+l_6+2l_9 \equiv 0 \mod 6$, $l_6=1$, $l_9$ odd and $2+2l_9
\equiv 0 \mod 6$, a contradiction.

If $l_8=1$, then $l_3=l_4=0$, $a_2=1+l_6+2+l_9 \equiv 0 \mod 2$,
$a_3=1+l_6+2+2l_9 \equiv 0 \mod 6$, $l_6=1$, $l_9$ even and $4+2l_9
\equiv 0 \mod 6$, a contradiction.

\smallskip
\noindent CASE 2.2.2: \,$l_{7} = 3$.

Then $l_1+l_2 = 1$, $a_1 \equiv l_3+l_5 \equiv 1 \mod 2$ (hence
$(l_1,l_2,l_3,l_5) \in \{(1,0,1,0), (1,0,0,1), (0,1,1,0),$ $
(0,1,0,1)\}$), and $a_2 \equiv l_2+l_4+l_6+l_9 \equiv 1 \mod 2$,
$a_3 \equiv l_1+l_5+l_6+2l_8+2l_9 \equiv 0 \mod 6$ and $a_4 \equiv
l_1+l_3+l_4+l_8 \equiv 0 \mod 6$.

If $(l_1,l_2,l_3,l_5) = (1,0,1,0)$, then $l_4+l_6+l_9 \equiv 1 \mod
2$, $1+l_6+2l_8+2l_9 \equiv 0 \mod 6$ and $2+l_4+l_8 \equiv 0 \mod
6$, hence $l_6 = 1$, $2 \t l_4+l_9$, $3 \t 1+l_8+l_9$ and $6 \t
2+l_4+l_8$, a contradiction.

If $(l_1,l_2,l_3,l_5) = (1,0,0,1)$, then $l_4+l_6+l_9 \equiv 1 \mod
2$, $2+l_6+2l_8+2l_9 \equiv 0 \mod 6$ and $1+l_4+l_8 \equiv 0 \mod
6$, hence $l_6 = 0$, $l_4=1$ and $l_8=4$. The first congruence forces $l_9$ to be even, whereas the second gives $l_9\equiv1\mod3$; this is impossible for $l_9\in[0,3]$. 

If $(l_1,l_2,l_3,l_5) = (0,1,1,0)$, then $l_4+l_6+l_9 \equiv 0 \mod
2$, $l_6+2l_8+2l_9 \equiv 0 \mod 6$ and $1+l_4+l_8 \equiv 0 \mod 6$,
hence $l_6=0$, $2 \t l_4+l_9$, $3 \t l_8+l_9$ and $6 \t 1+l_4+l_8$;
so $l_4=1$, $l_8=4$ and $l_9=2$, a contradiction.

If $(l_1,l_2,l_3,l_5) = (0,1,0,1)$, then $l_4+l_6+l_9 \equiv 0 \mod
2$, $1+l_6+2l_8+2l_9 \equiv 0 \mod 6$ and $l_4+l_8 \equiv 0 \mod 6$,
hence $l_6=1$, $2 \t 1+l_4+l_9$, $3 \t 1+l_8+l_9$ and $6 \t
l_4+l_8$, a contradiction.

\smallskip
2. In order to prove that $S$ is maximal zero-sum free 
we  introduce a further basis for $G$. We define
$e_5' = e_1 + e_5$ and observe that 
$(e_1,e_2,e_3,e_4,e_5')$ is  a basis of $G$. Then we have
\begin{align*}
	g_1&=e_1+e_3+e_4+e_5',\quad g_2=e_1+e_2+e_5',\quad g_3=e_1+e_4,\quad g_4=e_2+e_4,\\
	g_5&=e_1+e_3,\quad g_6=e_2+e_3,\quad g_7=e_2+2e_4+e_5',\\
	g_8&=2e_3+e_4,\quad g_9=e_2+2e_3,\quad g_{10}=e_5'.
\end{align*}
For any element $x_1e_1+x_2e_2+x_3e_3+x_4e_4+x_5e_5' \in G$, we use the compact coordinate notation
\begin{align*}
	(x_1,x_2,x_3,x_4,x_5)^t &:= x_1e_1+x_2e_2+x_3e_3+x_4e_4+x_5e_5'\\
	&=(e_1,e_2,e_3,e_4,e_5')\cdot (x_1,x_2,x_3,x_4,x_5)^t.
\end{align*}
We set
\[
S = S'  (e_5')^{6n-4} \quad \text{with} \quad S' = g_1 g_2 g_3 g_4 g_5 g_6 \, g_7^3 g_8^4 g_9^3 \,.
\]
We have 
\[
\Sigma^*\bigl((e_5')^{6n-4}\bigr) = \langle e_5' \rangle \setminus \{-e_5',-2e_5',-3e_5'\} \,,
\]
and for 
$T = g_3 g_4 g_5 g_6 \, g_8^4 g_9^3$ we have
\begin{align*}
	\Sigma^*(T)=
	\{ \left(\begin{array}{l|l|l|l|l|l|l|l|l|l|l|l|l|l}
		0 & 1 & 0 & 1 & 0 & 1 & 0 & 1 & 1 & 0 & 1 & 0 & 1 & 0\\
		0 & 0 & 1 & 0 & 1 & 1 & 0 & 1 & 1 & 1 & 0 & 1 & 0 & 0\\
		0 & 0 & 0 & 1 & 1 & 0 & 1 & 1 & 2 & 1 & 1 & 2 & 2 & 2\\
		0 & 1 & 1 & 0 & 0 & 2 & 1 & 1 & 0 & 2 & 2 & 1 & 1 & 2\\
		0 & 0 & 0 & 0 & 0 & 0 & 0 & 0 & 0 & 0 & 0 & 0 & 0 & 0
	\end{array}\right)
	\} \\
	+  \{
	\left(\begin{array}{l|l|l|l|l}
		0 & 0 & 0 & 0 & 0\\
		0 & 0 & 0 & 0 & 0\\
		0 & 2 & 4 & 0 & 2\\
		0 & 1 & 2 & 3 & 4\\
		0 & 0 & 0 & 0 & 0
	\end{array}\right)
	\}+  \{
	\left(\begin{array}{l|l|l|l}
		0 & 0 & 0 & 0\\
		0 & 1 & 0 & 1\\
		0 & 2 & 4 & 0\\
		0 & 0 & 0 & 0\\
		0 & 0 & 0 & 0
	\end{array}\right)
	\}.
\end{align*}
Therefore, $ \langle e_1, \ldots, e_4 \rangle^{\bullet} \setminus \Sigma (T)=\langle e_1, \ldots, e_4 \rangle \setminus \Sigma^* (T)$ is a subset of 
\begin{align*}
\{\left(\begin{array}{l|l|l|l|l|l}
	1 & 0 & 1 & 0 & 1 & 1\\
	0 & 1 & 0 & 0 & 1 & 1\\
	0 & 5 & 5 & 5 & 5 & 4\\
	0 & 1 & 1 & 5 & 5 & 0\\
	0 & 0 & 0 & 0 & 0 & 0
\end{array}\right)
\} \subseteq 
\left(\begin{array}{l}
	1 \\ 0 \\ 0 \\ 0 \\ 0 
\end{array}\right)
+\left\langle \left(\begin{array}{l}
	1 \\ 1 \\ 0 \\ 0 \\ 0 
\end{array}\right),
\left(\begin{array}{l}
	1 \\ 0 \\ 2 \\ 0 \\ 0 
\end{array}\right),
\left(\begin{array}{l}
	0 \\ 0 \\ 5 \\ 1 \\ 0 
\end{array}\right) \right\rangle.
\end{align*}

Let $G_1\cong C_2^2\oplus C_6^3$ have a basis $(e_1,e_2,e_3,e_4,e_0)$ whose elements have orders $2,2,6,6,6$, respectively.
We define a group homomorphism
\begin{align*}
f: G \to G_1,\quad x_1e_1+x_2e_2+x_3e_3+x_4e_4+x_5e_5' \longmapsto x_1e_1+x_2e_2+x_3e_3+x_4e_4+x_5 e_0.
\end{align*}

Since $\Sigma(S')\subseteq \{0,e_5',...,5e_5'\}+ \langle e_1,...,e_4\rangle$, the restriction of $f$ to $\Sigma(S')$ is injective, and 
\[
f\left( (e_1,e_2,e_3,e_4,e_5')\cdot (x_1,x_2,x_3,x_4,x_5)^t\right) =(e_1,e_2,e_3,e_4,e_0)\cdot (x_1,x_2,x_3,x_4,x_5)^t
\]
with $x_5\in [0,5]$. For the subset
$$C = (\{0,e_5',...,5e_5'\}+ \langle e_1,...,e_4\rangle)^{\bullet} \setminus \Sigma(S') \ \subseteq \ G^{\bullet}$$
we obtain that
\begin{align*}
f(C)=
\{(e_1,e_2,e_3,e_4,e_0)\cdot 
\left(\begin{array}{l|l|l|l|l|l|l|l|l|l|l|l|l|l}
	0 & 0 & 0 & 1 & 0 & 1 & 0 & 0 & 1 & 1 & 1 & 1 & 1 & 1\\
	0 & 0 & 0 & 1 & 1 & 0 & 1 & 1 & 0 & 0 & 1 & 1 & 1 & 0\\
	0 & 0 & 5 & 5 & 0 & 0 & 5 & 5 & 1 & 5 & 0 & 4 & 5 & 0\\
	0 & 0 & 5 & 5 & 2 & 2 & 1 & 1 & 1 & 1 & 0 & 0 & 1 & 0\\
	4 & 5 & 0 & 0 & 5 & 5 & 0 & 1 & 5 & 0 & 5 & 0 & 5 & 0
\end{array}\right)
\} = G_1^{\bullet} \setminus \Sigma (f(S')) \,,
\end{align*}

whence

\begin{align*}
C=
\{(e_1,e_2,e_3,e_4,e_5')\cdot 
\left(\begin{array}{l|l|l|l|l|l|l|l|l|l|l|l|l|l}
	0 & 0 & 0 & 1 & 0 & 1 & 0 & 0 & 1 & 1 & 1 & 1 & 1 & 1\\
	0 & 0 & 0 & 1 & 1 & 0 & 1 & 1 & 0 & 0 & 1 & 1 & 1 & 0\\
	0 & 0 & 5 & 5 & 0 & 0 & 5 & 5 & 1 & 5 & 0 & 4 & 5 & 0\\
	0 & 0 & 5 & 5 & 2 & 2 & 1 & 1 & 1 & 1 & 0 & 0 & 1 & 0\\
	4 & 5 & 0 & 0 & 5 & 5 & 0 & 1 & 5 & 0 & 5 & 0 & 5 & 0
\end{array}\right)
\},
\end{align*}

We now claim that $f$ induces an injective map from the set 
\begin{align*}
A := G^\bullet \setminus \Sigma\left(S' g_{10}^{6n-6}\right)
\end{align*}
to the set
\begin{align*}
B := G_1^\bullet \setminus \Sigma\left(f(S')\right) = f(C).
\end{align*}
We  recall that $g_{10} = e_5'$, whence $\Sigma^* \big( g_{10}^{6n-6} \big) = \{0, e_5', 2e_5', \dots, (6n-6)e_5'\}$, and we continue with the following two assertions.
\begin{enumerate}
\item[{\bf A1.}\,] $f (A) \subseteq B$.

\item[{\bf A2.}\,] $f$ is injective on $A$.
\end{enumerate}

\smallskip

{\it Proof of \,{\bf A1}}. Let $g = (x_1, x_2, x_3, x_4, x_5)^t \in A$. Since
	$g \notin \Sigma(S' g_{10}^{6n-6})$, we claim that
	\begin{equation}\label{eq:not-in-SigmaSprime}
	g + m e_5' \notin \Sigma(S')
	\quad \text{for all } m \in [6, 6n].
	\end{equation}
	Indeed, if $g + m e_5' \in \Sigma(S')$ for some $m \in [6, 6n]$, then
	$-m e_5' \in \Sigma^*(g_{10}^{6n-6})$ (since
	$-m \equiv 6n - m \in [0, 6n-6] \mod{6n}$), and rearranging gives
	$g = (g + m e_5') + (-m e_5') \in \Sigma(S' g_{10}^{6n-6})$,
	contradicting $g \in A$. In particular, taking $m = 6k$ for
	$k \in [1,n]$, we obtain $g + 6k e_5' \notin \Sigma(S')$.
	
	Let $s = x_5 \bmod 6$, so that $6 \mid (6n + s - x_5)$ and
	$6n + s - x_5 \in [1, 6n]$. Set $k_0 = (6n + s - x_5)/6$. Then
	$g + 6k_0 e_5'$ has $e_5'$-coordinate $s \in [0, 5]$, so
	$g + 6k_0 e_5' \in
	\bigl(\{0, e_5', \dots, 5e_5'\} + \langle e_1, \dots, e_4\rangle
	\bigr)^{\bullet}.$
	
	On the other hand, $g + 6k_0 e_5' \notin \Sigma(S')$,
	hence $g + 6k_0 e_5' \in C$. Since $f(e_5') = e_0$ and
	$\ord(e_0) = 6$, we have
	$f(g) = f(g) + k_0 (6 e_0) = f(g + 6k_0 e_5') \in f(C) = B$,
	as required.
	
{\it Proof of \,{\bf A2}}.	 Assume, to the contrary, that there exist distinct $g, h \in A$ with
	$f(g) = f(h)$. Then $h - g \in \ker(f) = \langle 6 e_5' \rangle$, so
	$h = g + 6k_1 e_5'$ for some integer $k_1$ with $1 \leq k_1 \leq n-1$.
	
	Applying~\eqref{eq:not-in-SigmaSprime} to $g$, we have
	$g + j e_5' \notin \Sigma(S')$ for every
	$j \in [6,6n]$.
	Applying~\eqref{eq:not-in-SigmaSprime} to $h$ and using
	$h + m e_5' = g + (6k_1 + m) e_5'$ yields
	$g + j e_5' \notin \Sigma(S')$ for every
	$j \in \mathbb{Z}/6n\mathbb{Z} \setminus \{6k_1 + 1, \dots, 6k_1 + 5\}$.
	Since $1 \leq k_1 \leq n-1$, we have $6k_1 + 1 \geq 7$ and
	$6k_1 + 5 \leq 6n - 1$, so the two exempt sets $\{1, \dots, 5\}$ and
	$\{6k_1 + 1, \dots, 6k_1 + 5\}$ are disjoint in $\mathbb{Z}/6n\mathbb{Z}$.
	Combining the two conditions, $g + j e_5' \notin \Sigma(S')$ for
	\emph{every} $j \in \mathbb{Z}/6n\mathbb{Z}$, that is, $(g + \langle e_5' \rangle) \cap \Sigma(S') = \emptyset.$
	
	In particular, the element
	$g' := (x_1, x_2, x_3, x_4, 2)^t = g + (2 - x_5) e_5'$
	lies in $g + \langle e_5'\rangle$ and satisfies $g' \notin \Sigma(S')$.
	Since its $e_5'$-coordinate is $2 \in [0, 5]$, we have
	$g' \in C$. However, inspection of the explicit description
	of $C$ above shows that the $e_5'$-coordinates of its
	elements (the last row of the matrix) take values only in
	$\{0, 1, 4, 5\}$, never $2$. This contradicts $g' \in C$.
	Therefore $f$ is injective on $A$.

\medskip
Combining {\bf A1} and {\bf A2} with
$
\Sigma^*\bigl((e_5')^{6n-6}\bigr)
=\{0,e_5',\ldots,(6n-6)e_5'\},
$
and translating in the $e_5'$-coordinate, we obtain the following
list of the only possible lifts of the elements of set $B$.
\begin{align*}
G^{\bullet} \setminus \Sigma (S'g_{10}^{6n-6}) \subseteq
\{\left(\begin{array}{l|l|l|l|l|l|l|l|l|l|l|l|l|l}
	0 & 0 & 0 & 1 & 0 & 1 & 0 & 0 & 1 & 1 & 1 & 1 & 1 & 1\\
	0 & 0 & 0 & 1 & 1 & 0 & 1 & 1 & 0 & 0 & 1 & 1 & 1 & 0\\
	0 & 0 & 5 & 5 & 0 & 0 & 5 & 5 & 1 & 5 & 0 & 4 & 5 & 0\\
	0 & 0 & 5 & 5 & 2 & 2 & 1 & 1 & 1 & 1 & 0 & 0 & 1 & 0\\
	6n-2 & 6n-1 & 0 & 0 & 6n-1 & 6n-1 & 0 & 1 & 6n-1 & 0 & 6n-1 & 0 & 6n-1 & 0
\end{array}\right)
\} \,,
\end{align*}
whence $G^{\bullet} \setminus \Sigma (S'g_{10}^{6n-5})\subseteq \{(0,0,0,0,6n-1),(0,1,5,1,1)\}$, and thus $\Sigma(S)=G^{\bullet}$. 

\bigskip
3. By Lemma \ref{6.2}, it suffices to verify that  $\nu (S) \le |S|-1$. In order to do so, we prove that
$|S| - 1$ satisfies \eqref{def-nu(S)}. Let $T \in \mathcal F (G)$ with $|T|=|S|-1$, say $S = gT$ for some $g \in \supp (S)$, and $g\notin 2G$.  Note that $\ord(g) \in \{6, 6n\}$  and  $\sigma(S) = e_2 + 5e_3 + e_4+e_5'$ has order $6n$. We distinguish two cases.

\smallskip
\noindent
CASE 1: $g \neq g_1$.

Using the same reduction of the $e_5'$-coordinate modulo
$\langle 6e_5'\rangle$ as in Part 2, it suffices to carry out the
verification in $
G/\langle 6e_5'\rangle\cong C_2^2\oplus C_6^3.$
A direct computation in this quotient, together with the injective lifting argument in Part 2 applied to $Sg^{-1}$, shows that 
\[
	G^{\bullet} \setminus \Sigma(Sg^{-1})=\{-g, \sigma (S) \} \,.
\]
If $\ord(g)=6n$, then $g\in\{g_2,g_7,g_{10} \}$ whence $g \in e_5'+\langle e_1,...,e_4\rangle$. Thus $\sigma (S) \notin \langle g+ \sigma (S) \rangle \subseteq \langle e_1,...,e_4,2e_5'\rangle$, whence $\{-g, \sigma (S) \} \subseteq e_5' + \langle e_1,...,e_4,2e_5'\rangle$.
		
If $\ord(g)=6$, then $g+\sigma (S) \in e_5'+\langle e_1,...,e_4\rangle$ and $\ord(g+ \sigma (S))=6n$. Since $\langle e_1,...,e_4\rangle \cap \langle g+\sigma (S) \rangle=\{0\}$ and $g\ne 0$, it follows that  $-g\notin \langle g+\sigma (S) \rangle$. Thus Lemma \ref{2.3}(3) implies the assertion similarly.

\smallskip
\noindent
CASE 2: $g=g_1$. 

The proof of Part 2 (or direct computation) shows that
	\begin{align*}
	\begin{aligned}
	G^{\bullet} \setminus \Sigma(Sg_1^{-1})
	&=
	\bigl\{
	e_1+5e_3+5e_4+(6n-1)e_5',\, e_1+5e_3+e_4+e_5',\\
	&\qquad e_2+5e_3+5e_4+(6n-1)e_5',\, e_2+5e_3+e_4+e_5'
	\bigr\}.
	\end{aligned}
	\end{align*}
Thus we obtain that
\[
	G^{\bullet} \setminus \Sigma(Sg_1^{-1})\subseteq
	g_1 + \langle e_1,e_2,e_3, e_4,2e_5' \rangle \,. \qedhere
\]
\end{proof}

\smallskip
Let $G = C_n \oplus C_n$ with $n \ge 2$. Then $\mathsf d (G) = \mathsf d^* (G)=2n-2$, and
every minimal zero-sum sequence $S \in \mathcal F (G)$  with $|S| = \mathsf D^* (G)$ has the property that $\mathsf h (S)=n-1$ (\cite[Theorem 4.2.1]{Ge-Gr-Zh26a}). Lemma \ref{6.3} shows that, for every $r \ge 2$ there are minimal zero-sum sequences $S \in \mathcal F (C_n^r)$ with $|S|= \mathsf D^* (C_n^r)$ and with $\mathsf h (S)=n-1$. We are not aware of a minimal zero-sum sequence $S \in \mathcal F (C_n^3)$ in the literature with $|S|=\mathsf D^* (C_n^3)$ and with $\mathsf h (S) \le n-2$. Below we provide an example of such a sequence $S$ and determine $\nu (S)$. Note that $\mathsf d (C_n^3) = \mathsf d^* (C_n^3)$ if $n$ is a prime power, and if $n = 2p^k$ or $n = 3p^k$, with $k \in \N$ and $p$ a prime (\cite[Section 4]{Sc11b}). 

\begin{theorem} \label{6.6}
	Let $G = C_n \oplus C_n \oplus C_n$ with $n \ge 2$ and let $(e_1, e_2, e_3)$ be a basis of $G$ with $\ord (e_i)=n$ for $i \in [1,3]$.
	\begin{enumerate}
		\item The sequence 
		\[
		U =e_1^{n-2}\prod_{i=1}^n (a_ie_1+e_2)(b_ie_1+e_3) \in \mathcal F (G) \,,
		\]
		where $a_1, b_1, \ldots, a_n, b_n \in [0,n-1]$ and  $\sum_{i=1}^n a_i \equiv \sum_{i=1}^n b_i \equiv 1 \mod{n}$, is a minimal zero-sum sequence with $|U| = \mathsf D^* (G)$.
		
		\item Let $U \in \mathcal A (G)$ be as in Part 1, and suppose that  $\mathsf h(  U)=n-2>1$. If $S \in \mathcal F (G)$ with $S \mid U$ and $|S|=|U|-1$, then $\nu (S) = |S|-1$.
	\end{enumerate}
\end{theorem}

\begin{proof}
	1. Clearly, $U$ has sum $\sigma (U)=0$. Let $V$ be a nontrivial zero-sum subsequence of $U$. If $a_ie_1+e_2\mid V$ for some $i\in [1,n]$, then $\prod_{i=1}^n (a_ie_1+e_2)\mid V$. If $b_ie_1+e_3\mid V$ for some $i\in [1,n]$, then $\prod_{i=1}^n (b_ie_1+e_3)\mid V$. Since $\sum_{i=1}^n a_i \equiv \sum_{i=1}^n b_i \equiv 1 \mod{n}$, it follows that  $\sigma(\prod_{i=1}^n (a_ie_1+e_2))=\sigma(\prod_{i=1}^n (b_ie_1+e_3))=e_1$, whence $V=U$. 
	
	2. The hypothesis $\mathsf h(U)=n-2>1$ implies that $n\ge 4$. Let $S \in \mathcal F (G)$ with $S\mid U$ and $|S|=|U|-1$. Since $\sum_{i=1}^n a_i \equiv \sum_{i=1}^n b_i \equiv 1 \mod{n}$ and $\mathsf h(  U)=n-2$, neither $a_1,\ldots,a_n$ nor $b_1,\ldots,b_n$ is constant. Without loss of generality, we may assume that $a_1 \notin \{ a_{n-1}, a_n \}$ and $b_1 \notin  \{ b_{n-1}, b_n \}$.
	
	In the first step, we show that  $S$ is  maximal zero-sum free. Assume to the contrary that $S$ is not a maximal zero-sum free sequence. Let $\alpha=xe_1+ye_2+ze_3$, with $x,y,z \in [0, n-1]$, such that $-\alpha\in G^{\bullet}\setminus \Sigma (S)$. 
	
	\smallskip
	\noindent
	CASE 1: $S=Ue_1^{-1}$.
	
We distinguish three cases.
	\begin{itemize}
		\item[(i)] $y=z=0$. Since $-\alpha\in \langle e_1 \rangle^{\bullet} \subseteq  \Sigma (S)$, this is a contradiction.
		
		\item[(ii)] $y\ne 0$ or $z\ne 0$, we  discuss one case in detail, $y\ne 0$ and $z=0$. Then $\alpha \prod_{i=1}^n (a_ie_1+e_2)$ has at least 2 subsequences with different sums inside $\langle e_1 \rangle$. Since $\langle e_1 \rangle^{\bullet}\setminus\{-e_1\} \subseteq \Sigma (e_1^{n-3}\prod_{i=1}^n (b_ie_1+e_3))$, there is a zero-sum subsequence of $S\alpha$, a contradiction.
		
		\item[(iii)] $y\ne 0$ and $z\ne 0$. Let $I_1= [1,n-y]$  and $J_1=[1,n-z]$, such that $\sigma(\alpha \prod_{i\in I_1} (a_ie_1+e_2) \prod_{j\in J_1} (b_je_1+e_3) ) = d_0\in \langle e_1 \rangle$. 
		Since $a_i$ and $b_i$ cannot be all the same for all $i\in [1,n]$, there exist $I_2\subsetneq [1,n]$ with $|I_2|=n-y$ and $J_2\subsetneq [1,n]$ with $|J_2|=n-z$, where $I_2\ne I_1$ and $J_2\ne J_1$. Hence
		$$\sigma( \prod_{i\in I_2} (a_ie_1+e_2)) -\sigma( \prod_{i\in I_1} (a_ie_1+e_2))= d_1\in \langle e_1 \rangle^{\bullet},$$
		$$\sigma( \prod_{j\in J_2} (b_je_1+e_3)) -\sigma( \prod_{j\in J_1} (b_je_1+e_3))= d_2\in \langle e_1 \rangle^{\bullet}.$$
		If $n\ge 4$ is even and $2d_1=2d_2=0$, then $d_0\notin \{ e_1,2e_1\}$ or $d_0+d_1 \notin \{ e_1,2e_1\}$. Hence there is a zero-sum subsequence of $S\alpha$. 
		
		In all remaining cases, $|\{d_0,d_0+d_1,d_0+d_2,d_0+d_1+d_2\}|\ge 3$.
		Thus there are at least 3 subsequences of $\alpha \prod_{i=1}^n (a_ie_1+e_2)(b_ie_1+e_3) $ with different sums inside $\langle e_1 \rangle$, and one of them $T$ with $\sigma(T)\in \langle e_1 \rangle\setminus \{ e_1,2e_1\}$. Then there is a zero-sum subsequence of $S\alpha$, a contradiction.
		
	\end{itemize}

	\smallskip
	\noindent
	CASE 2: $S\ne Ue_1^{-1}$,  say $S= U(a_ne_1+e_2)^{-1}$. 

Then $\langle e_1 \rangle^{\bullet} \subseteq \Sigma(e_1^{n-2}\prod_{i=1}^n (b_ie_1+e_3))$. We distinguish three cases.
	\begin{itemize}
		\item[(i)] $y=z=0$, then $-\alpha \in \langle e_1 \rangle^{\bullet} \subseteq \Sigma (S)$ and there is a zero-sum subsequence of $S\alpha$,  a contradiction.
		
		\item[(ii)] $z=0$ and $y\ne 0$. 
		Since $n-y\in[1,n-1]$, we have $\alpha \prod_{i=1}^{n-y} (a_ie_1+e_2)\in \langle e_1 \rangle$.
		So there is a zero-sum subsequence of $S\alpha$, a contradiction.
		
		\item[(iii)]$z\ne 0$. Let $I\subseteq [1,n-1]$ and $|I|\equiv n-y \mod{n}$. Let $J=[2,n-z] \cup \{ n\} \subseteq [2,n]$.  Then $\sigma( \prod_{j= 1}^{n-z} (b_je_1+e_3)) -\sigma( \prod_{j\in J} (b_je_1+e_3))=(b_n -b_1)e_1 \in \langle e_1 \rangle^{\bullet}.$ 
		Thus there are at least 2 subsequences
		$T_1= \alpha \prod_{i\in I} (a_ie_1+e_2) \prod_{j= 1}^{n-z} (b_je_1+e_3)$ and $T_2= \alpha \prod_{i\in I} (a_ie_1+e_2) \prod_{j\in J} (b_je_1+e_3)$ with different sums inside $\langle e_1 \rangle$, so we may assume $\sigma(T_1) \in \langle e_1 \rangle \setminus \{e_1 \}$. Then there is a zero-sum subsequence of $e_1^{n-2}T_1$, a contradiction.
\end{itemize}

\medskip	
	In the second step,  we prove that $\nu (S) = |S|-1$.
	Since $|S|-1 \le \nu (S)$ by Lemma \ref{6.2}, it suffices to verify that $|S| - 1$ satisfies \eqref{def-nu(S)}. Let $T \in \mathcal F (G)$ with
	$T\mid S$ and with $|T|=|S|-1=|U|-2$, say  $U=Sg=Tgh$ with $g, h \in G$.  Let $\varphi \colon G\to G/\langle e_1 \rangle \cong C_n^2$ be the canonical epimorphism. 
	
	If $ G^{\bullet}\setminus \Sigma (T) = \emptyset$, then the claim is clear. 
	Suppose that $ G^{\bullet}\setminus \Sigma (T) \ne \emptyset$.

	\smallskip
	\noindent
	CASE 1:  $g \ne e_1$ and $h \ne e_1$. 
	
	Then $g\mid \prod_{i=1}^n (a_ie_1+e_2)(b_ie_1+e_3)$, say  $g=a_1e_1+e_2$.
	
	\smallskip
	\noindent
	CASE 1.1:  $h\mid \prod_{i=2}^n (a_ie_1+e_2)$,  say $h=a_2e_1+e_2$. 
	
	We set $g'=xe_1+ye_2+ze_3$ with $x,y,z \in [0, n-1]$, where $-g'\in G^{\bullet}\setminus \Sigma (T)$, and we claim that $y=1$. 
	
	If $y>1$, then $g'\prod_{i=3}^{n-y+2} (a_ie_1+e_2)\mid Tg'$ and we define $g_0 :=\sigma(g'\prod_{i=3}^{n-y+2} (a_ie_1+e_2))\in \langle e_1,e_3 \rangle$. 
	
	If $y=0$, then we define $g_0:=g'$, and we observe that $g' \in \langle e_1,e_3 \rangle$.
	
	In both cases, $V:=g_0e_1^{n-2}\prod_{i=1}^n (b_ie_1+e_3) \in \mathcal F (\langle e_1,e_3 \rangle)$ is a sequence of length $|V|=2n-1=\mathsf{D}(\langle e_1,e_3 \rangle)$. Thus $V$ has a zero-sum subsequence, whence $Tg'$ has a zero-sum subsequence, a contradiction.
	
	Thus it follows that  $y=1$ and $G^{\bullet}\setminus \Sigma (T)\subseteq -e_2+ \langle e_1,e_3 \rangle$ with $(G:\langle e_1,e_3 \rangle)=n$. 
	
	\smallskip
	\noindent
	CASE 1.2: $h\mid \prod_{i=1}^n (b_ie_1+e_3)$,  say $h=b_1e_1+e_3$.

The same coordinate argument as in Part 1 shows that $Te_1$ is zero-sum free. Hence $-e_1$ and $\sigma(Te_1)=e_1-g-h=-(a_1+b_1-1)e_1-e_2-e_3$ do not belong to $\Sigma(T)$. Moreover, $-g,-h\notin\Sigma(T)$ because $U=Tgh$ is minimal zero-sum. We assert that
\[G^{\bullet}\setminus \Sigma (T)=\{-e_1,-g,-h, -(a_1+b_1-1)e_1-e_2-e_3\}.\] 
	Assume to the contrary that $G^{\bullet}\setminus \left(\Sigma (T) \cup \{ -e_1,-g,-h, -(a_1+b_1-1)e_1-e_2-e_3\}\right)\ne \emptyset$, say $-g'\in G^{\bullet}\setminus \left(\Sigma (T) \cup \{ -e_1,-g,-h, -(a_1+b_1-1)e_1-e_2-e_3\}\right)$. We distinguish five cases.

	
	\begin{itemize}
		\item[(i)] $g'=xe_1$. Then $x=1$.
		
		\item[(ii)] $g'=xe_1+e_2$ with $x\ne a_1$, or
		$g'=xe_1+e_3$ with $x\ne b_1$, we  discuss one case in detail. Then $x+\sum_{i=2}^n a_i \not\equiv  1 \mod{n}$ and $e_1^{n-2}g'\prod_{i=2}^{n} (a_ie_1+e_2)$ has a zero-sum subsequence, a contradiction.
		
		\item[(iii)]  $g'=xe_1+ye_2$ with $y> 1$. 
		
		Since $\mathsf h(U)=n-2$, the coefficients $a_2,\ldots,a_n$ are not all equal; after relabeling, assume $a_2\ne a_n$.
		
		Let $R=g'\prod_{i=2}^{n-y+1} (a_ie_1+e_2)\mid S$ and $R'=g'(a_ne_1+e_2)\prod_{i=3}^{n-y+1} (a_ie_1+e_2)\mid S$. Then $\sigma(R),\sigma(R') \in \langle e_1 \rangle$. If $\sigma(R)\ne  e_1$, then $e_1^{n-2}R$ has a zero-sum subsequence. If $\sigma(R)=  e_1$, then $\sigma(R')= (a_n-a_2+1) e_1$
		and $e_1^{n-2}R'$ has a zero-sum subsequence, a contradiction.
		
		\item[(iv)] $g'=xe_1+ye_2+ze_3$ with $z\ge 1$ and $(y,z)\ne (1,1)$. 
		
		Then $y>1$ or $z>1$; by symmetry, assume $y>1$, and after relabeling, assume $a_2\ne a_n$.
		
		Let $R=g'\prod_{i=2}^{n-y+1} (a_ie_1+e_2)\prod_{j=2}^{n-z+1} (b_je_1+e_3)\mid S$ and $R'=g'(a_ne_1+e_2)\prod_{i=3}^{n-y+1} (a_ie_1+e_2)\prod_{j=2}^{n-z+1} (b_je_1+e_3)\mid S$. Then $\sigma(R),\sigma(R') \in \langle e_1 \rangle$. If $\sigma(R)\ne  e_1$, then $e_1^{n-2}R$ has a zero-sum subsequence. If $\sigma(R)=  e_1$, then $\sigma(R')= (a_n-a_2+1) e_1$
		and $e_1^{n-2}R'$ has a zero-sum subsequence, a contradiction.
		
		\item[(v)] $g'=xe_1+e_2+e_3$. Then $R:=g'\prod_{i=2}^n (a_ie_1+e_2)(b_ie_1+e_3)$ with $\sigma(R)=(2-a_1-b_1+x)e_1\in \langle e_1 \rangle$. Since $e_1^{n-2}R$ is zero-sum free, $x\equiv a_1+b_1-1\mod n$.
	\end{itemize}
	
	Thus only (i), (v), or $g'\mid gh$ can occur, and
	\[
	G^{\bullet}\setminus \Sigma (T)=\{-e_1,-g,-h, -(a_1+b_1-1)e_1-e_2-e_3\}\subseteq -e_1+ \langle e_1-g,e_1-h \rangle
	\] 
	with $(G:\langle e_1-g, e_1-h \rangle)=n$.
	
	\smallskip
	\noindent
	CASE 2:  $g=e_1$ and $h \ne e_1$.
	
	Then $h\mid \prod_{i=1}^n (a_ie_1+e_2)(b_ie_1+e_3)$, say  $h=a_1e_1+e_2$.

	Since $Te_1=Uh^{-1}$ is zero-sum free and $\sigma(Te_1)=-h$, both $-e_1$ and $-h$ lie outside $\Sigma(T)$, and we claim
	\[
	\{-e_1,-h\}= G^{\bullet}\setminus \Sigma (T).
	\]
	Assume to the contrary that $G^{\bullet}\setminus \left(\Sigma (T) \cup \{ -e_1,-h\}\right)\ne \emptyset$, say $-g'\in G^{\bullet}\setminus \left(\Sigma (T) \cup \{ -e_1,-h\}\right)$ and $h'=-\sigma(g'T)$.

 Since $Tg'$ is zero-sum free and $h'=-\sigma(Tg')$, the sequence $Tg'h'$ is minimal zero-sum. If there is some $\beta\in[0,n-1]$ such that, for every $i\in[1,n]$, either $b_i\equiv\beta\mod n$ or $b_i\equiv\beta+1\mod n$, then the number of indices $i$ satisfying $b_i\equiv\beta+1\mod n$ lies in $[2,n-2]$ because $\mathsf h(U)=n-2$, contradicting $\sum_{i=1}^n b_i\equiv1\mod n$.

Since $n\ge4$, the preceding paragraph implies that there exist $i,j\in[1,n]$ such that $b_i-b_j\not\equiv0,\pm1\mod n$. Let  $g'\in\langle e_1\rangle+ye_2+ze_3$ with $y,z\in[0,n-1]$.
	
Assume to the contrary that $z\ne0$. Choose no term $a_ke_1+e_2$ if $y=0$, and otherwise fix any $n-y$ of the terms $a_ke_1+e_2$ with $k\in[2,n]$. Together with $g'$ and any $n-z$ of the terms $b_ke_1+e_3$, these terms form a subsequence of $Tg'$ whose sum belongs to $\langle e_1\rangle$. This sum lies in $\{e_1,2e_1\}$, because otherwise it would be zero or could be completed to zero using at most $n-3$ copies of $e_1$. We may choose the $n-z$ terms so that $b_ie_1+e_3$ is chosen and $b_je_1+e_3$ is not. Replacing the former by the latter changes the sum by $(b_j-b_i)e_1\notin\{0,e_1,-e_1\}$, so the two sums cannot both belong to $\{e_1,2e_1\}$, a contradiction. 
	
Thus we obtain that $z=0$. Let $W=\prod_{i=1}^n(b_ie_1+e_3)e_1^{n-3}$. Since $W\mid T$ and $\sigma(\prod_{i=1}^n(b_ie_1+e_3))=e_1$, we have $\{e_1,\ldots,(n-2)e_1\}\subseteq\Sigma(W)\subseteq\Sigma(T)$ and $\sigma(W)=(n-2)e_1$. If $y=0$, then $g'\in\langle e_1\rangle$, and $-g'\notin\Sigma(T)$ and $-g'\ne-e_1$, contradicting the above inclusions. If $y=1$, then $h'=e_1+h-g'\in\langle e_1\rangle$. The minimality of $Tg'h'$ gives $-h'\notin\Sigma(T)$, so the above inclusion implies $h'=e_1$, whence $g'=h$,  a contradiction.
	
Thus we obtain that $y\in[2,n-1]$. For every $I\subseteq[2,n]$ with $|I|=n-y$, set $V=g'\prod_{k\in I}(a_ke_1+e_2)$. Then $V\mid Tg'W^{-1}$ and $\sigma(V)\in\langle e_1\rangle^\bullet$, since $Tg'$ is zero-sum free. If $\sigma(V)\in\{2e_1,\ldots,(n-1)e_1\}$, then $-\sigma(V)\in\Sigma(W)$, and hence $V$ can be completed to a zero-sum subsequence of $Tg'$ by a subsequence of $W$. Hence $\sigma(V)=e_1$ for every choice of the sets $I$. 
		
	Since $\mathsf h(U)=n-2$, the integers $a_2,\ldots,a_n$ cannot all be equal. Thus there exist $i,j\in[2,n]$ with $a_i\ne a_j$. Since $1\le n-y\le n-2$, there exist $i\in I$, and $j\notin I$ for some $I$. The preceding conclusion gives $\sigma(V)=\sigma(V(a_je_1+e_2)(a_ie_1+e_2)^{-1})=e_1$, whereas $\sigma(V(a_je_1+e_2)(a_ie_1+e_2)^{-1})-\sigma(V)=(a_j-a_i)e_1\ne0$, a contradiction.

	Thus, we infer that $G^{\bullet}\setminus \Sigma (T)=\{-e_1,-h \}\subseteq -e_1+ \langle e_1-h, e_3\rangle$ with $(G:\langle e_1-h, e_3 \rangle)=n$.
	
	\smallskip
	\noindent
	CASE 3:  $g=h=e_1$.
	
	Since $Te_1$ is zero-sum free and $-\sigma(Te_1)=e_1$, we claim
	\[
	\{-e_1\}= G^{\bullet}\setminus \Sigma (T).
	\]
	Assume to the contrary that $G^{\bullet}\setminus \left(\Sigma (T) \cup \{ -e_1\}\right)\ne \emptyset$, say $-g'\in G^{\bullet}\setminus \left(\Sigma (T) \cup \{ -e_1\}\right)$ and set $h' = - \sigma (Tg')$.  Then $g'+h' =2e_1$, and we either have $g', h' \notin \langle e_1 \rangle$ or $g', h' \in \langle e_1 \rangle^{\bullet}$. 
	
	If $g', h' \in \langle e_1 \rangle^{\bullet}$, then $g'\ne e_1$ by $-g'\in G^{\bullet}\setminus \left(\Sigma (T) \cup \{ -e_1\}\right)$.  Thus $g', h' \in \{3e_1, \ldots, (n-1)e_1 \}$. If $g'=h'=3e_1$ and $n=4$, then $g'\prod_{i=1}^n (a_ie_1+e_2)$ is a zero-sum subsequence. Otherwise, $g' h'e_1^{n-4}$ has a zero-sum subsequence. So we assume $g', h' \notin \langle e_1 \rangle$.
	
	Let $W=g'\prod_{i=2}^n (a_ie_1+e_2)(b_ie_1+e_3)$ with $|\varphi(W)|= 2n-1= \mathsf{D}(G/\langle e_1 \rangle)$. Then $W$ has a subsequence $R$ 
	where $g'\mid R$ and $\sigma(R)=me_1\in \langle e_1 \rangle$ for some $m\in[0,n-1]$. Indeed, $R$ must contain $g'$, since $e_2^{n-1}e_3^{n-1}$ is zero-sum free in $G/\langle e_1\rangle$. So $\sigma(Wh'(a_1e_1+e_2)
	(b_1e_1+e_3)R^{-1})=\sigma(Ue_1^{-(n-4)})- \sigma(R)=(4-m)e_1$. If $\sigma(R)=0$, then $R$ is a zero-sum subsequence of $Tg'h'$. Otherwise, if $\sigma(R)=me_1\notin\{e_1,2e_1,3e_1\}$, then $m\in[4,n-1]$ and $Re_1^{n-m}\mid Tg'h'$ is zero-sum. Both alternatives are contradictions.
	
	So we may assume that $m\in[1,3]$. If $y\ne0$ and there is no $\alpha\in[0,n-1]$ such that, for every $i\in[1,n]$, one has $a_i\equiv\alpha\mod n$, $a_i\equiv\alpha+1\mod n$, or $a_i\equiv\alpha+2\mod n$, choose the required $n-y$ terms in four ways, changing only four preselected terms. The four resulting sums in $\langle e_1\rangle$ are not all contained in $\{e_1,2e_1,3e_1\}$; hence one is zero or can be completed to zero using at most $n-4$ copies of $e_1$, a contradiction. The same argument applies to the $b_j$ when $z\ne0$.
	
	Thus, from now on, whenever $y\ne0$ we fix such an $\alpha\in[0,n-1]$, and whenever $z\ne0$ we fix such a $\beta\in[0,n-1]$.
	
	If $\alpha$ (or $\beta$) or $\alpha+2$ (or $\beta+2$) does not appear or both of them appear only once, since $\mathsf h(U)=n-2$, a contradiction to $\sum_{i=1}^n a_i \equiv 1 \mod{n}$ (or $b_j$ are the same).
	For each relevant coefficient family, neither endpoint residue class can be absent, and the two endpoint residue classes cannot both have multiplicity one. Indeed, if one endpoint residue class is absent, the coefficients lie in two consecutive residue classes and the multiplicity of the latter class lies in $[2,n-2]$, contradicting that the sum of the coefficients is $1$ modulo $n$; if both endpoint residue classes occur once, the sum of the coefficients is $0$ modulo $n$, again a contradiction.

After relabeling the relevant terms, we may therefore assume $a_1\equiv\alpha\mod n$, $a_n\equiv\alpha+2\mod n$, and either $a_2\equiv\alpha\mod n$ or $a_2\equiv\alpha+2\mod n$ whenever $y\ne0$, and analogously $b_1\equiv\beta\mod n$, $b_n\equiv\beta+2\mod n$, and either $b_2\equiv\beta\mod n$ or $b_2\equiv\beta+2\mod n$ whenever $z\ne0$. Interchanging $g'$ and $h'$ if necessary, there are the  following four cases.
	
	\begin{itemize}
		\item[(i)] $g'=xe_1+(n-1)e_2$ and $h'=(2-x)e_1+e_2$. 
		If $x\equiv1-\alpha\mod n$, then $R=h'\prod_{i=1}^{n-1} (a_ie_1+e_2)$ is a zero sum subsequence, a contradiction. 
		If $x\not\equiv1-\alpha\mod n$, then $\sigma(g'(a_ie_1+e_2))\notin \{e_1,2e_1,3e_1\}$ for some $i\in [1,n]$. So $R'=g'(a_ie_1+e_2)e_1^{n-4}$ has a zero sum subsequence, a contradiction.
		
		\item[(ii)] $g'=xe_1+(n-1)e_3$ and $h'=(2-x)e_1+e_3$. 
		If $x\equiv1-\beta\mod n$, then $R=h'\prod_{i=1}^{n-1} (b_ie_1+e_3)$ is a zero sum subsequence, a contradiction.
		If $x\not\equiv1-\beta\mod n$, then $\sigma(g'(b_ie_1+e_3))\notin \{e_1,2e_1,3e_1\}$ for some $i\in [1,n]$. So $R'=g'(b_ie_1+e_3)e_1^{n-4}$ has a zero sum subsequence, a contradiction.
		
		\item[(iii)] $g'=xe_1+ye_2$ with $y\in[2,n-2]$, or, symmetrically, $g'=xe_1+ze_3$ with $z\in[2,n-2]$. 
		For every $I\subseteq[1,n]$ with $|I|=n-y$, the zero-sum freeness of $Tg'$ and $Th'$ applied to $I$ and its complement gives $x+\sum_{i\in I}a_i\equiv1\mod n$ or $x+\sum_{i\in I}a_i\equiv2\mod n$, because the complementary coefficient is congruent to $3-(x+\sum_{i\in I}a_i)$ modulo $n$. On the other hand, choosing four terms as above gives three such fixed-cardinality sums with distinct residue classes modulo $n$, a contradiction; the $e_3$-case is symmetric.
	
		\item[(iv)] $g'=xe_1+ye_2+ze_3$ with $y,z\ne 0$. Thus $R=g'\prod_{i=1}^{n-y} (a_ie_1+e_2)\prod_{j=1}^{n-z}(b_je_1+e_3)$ with $\sigma(R) \in \langle e_1 \rangle$. 
		
Choose $k\in[0,n-1]$ such that $\sigma(R)=ke_1$. If $\sigma(R)\notin \{e_1,2e_1,3e_1\}$, then either $k=0$ and $R$ itself is zero-sum, or $k\in[4,n-1]$ and $Re_1^{n-k}$ is a zero sum subsequence of $g'T$.
		
		If $\sigma(R)\in\{e_1,2e_1,3e_1\}$ and $n\ge5$, then the two valid exchanges $R_1=R(a_1e_1+e_2)^{-1}(a_ne_1+e_2)$ and $R_2=R(a_1e_1+e_2)^{-1}(b_1e_1+e_3)^{-1}(a_ne_1+e_2)(b_ne_1+e_3)$ change the sum by $2e_1$ and $4e_1$, respectively. Among $\sigma(R)$, $\sigma(R_1)$, and $\sigma(R_2)$, at least one is zero or lies outside $\{e_1,2e_1,3e_1\}$, and hence yields a zero-sum subsequence, a contradiction.
		
		If $n=4$, then $\sum_{i=1}^4a_i\equiv1\mod 4$ gives $3\alpha+a_3\equiv1 \mod 2$, and hence $a_3\equiv\alpha+1\mod 4$; similarly, $b_3\equiv\beta+1\mod 4$. For each number of selected terms in $[1,3]$, the residue classes modulo $4$ of the corresponding fixed-cardinality subsums of $a_1,\ldots,a_4$, and likewise of $b_1,\ldots,b_4$, therefore form a subset of $C_4$ of size at least three. The sum of two subsets of $C_4$ of size at least three is all of $C_4$; thus one can select respectively $4-y$ and $4-z$ terms whose coefficients cancel the $e_1$-coordinate of $g'$, yielding a zero-sum subsequence of $g'T$.
		\qedhere
	\end{itemize}
\end{proof}

\bigskip
\noindent
{\bf Acknowledgements.} We thank David J. Grynkiewicz for his help with the proof of Theorem \ref{6.4}.

\providecommand{\bysame}{\leavevmode\hbox to3em{\hrulefill}\thinspace}
\providecommand{\MR}{\relax\ifhmode\unskip\space\fi MR }
\providecommand{\MRhref}[2]{%
  \href{http://www.ams.org/mathscinet-getitem?mr=#1}{#2}
}
\providecommand{\href}[2]{#2}

\end{document}